\documentclass[a4paper,twoside,10pt]{article}
\usepackage[a4paper,left=3cm,right=3cm, top=3cm, bottom=3cm]{geometry}

\usepackage{amsmath}
\usepackage{amsthm}
\usepackage{amssymb}
\usepackage{cite}
\usepackage[author={Lorenzo}]{pdfcomment}
\usepackage{soul}
\usepackage{bm}

\theoremstyle{plain}
\newtheorem{thm}{Theorem}[section]
\newtheorem{cor}[thm]{Corollary}
\newtheorem{lem}[thm]{Lemma}
\newtheorem{rem}[thm]{Remark}
\newtheorem{prop}[thm]{Proposition}

\newcommand{\Nbb}{\mathbb N}
\newcommand{\Rbb}{\mathbb R}
\newcommand{\nbf}{\mathbf n}
\newcommand{\nbfGamma}{\nbf_\Gamma}
\newcommand{\Norm}[2]{\Vert #1 \Vert_{#2}}
\newcommand{\SemiNorm}[2]{\vert #1 \vert_{#2}}
\DeclareMathOperator{\curl}{\nabla\times}

\DeclareMathOperator{\curlh}{\nabla_\h\times}
\let\div\relax
\DeclareMathOperator{\div}{\nabla\cdot}
\DeclareMathOperator{\dt}{\partial_t}
\newcommand{\fbf}{\mathbf f}
\newcommand{\h}{h}
\newcommand{\Qbbh}{\mathbb Q_\h}
\newcommand{\Ibb}{\mathbb I}
\newcommand{\chih}{\chi_\h}
\newcommand{\ubf}{\mathbf u}
\newcommand{\Bbf}{\mathbf B}
\newcommand{\Ebf}{\mathbf E}
\newcommand{\jbf}{\mathbf j}
\newcommand{\Hbf}{\mathbf H}
\newcommand{\PiNcalu}{\PiNcal \mathbf u}
\newcommand{\RTcal}{\mathcal{RT}}

\newcommand{\Pitildez}{\widetilde\Pi^{0,\RTcal}_k}
\newcommand{\BbfI}{\Pitildez\mathbf B}
\newcommand{\PiNcal}{\Pi^{\mathcal N}_k}
\newcommand{\PiNcalE}{\PiNcal \mathbf E}
\newcommand{\EbfI}{\mathbf E_I}
\newcommand{\jbfI}{\mathbf j_I}

\DeclareMathOperator{\eubf}{\mathbf{e}^{\mathbf u}_\h}
\DeclareMathOperator{\eBbf}{\mathbf{e}^{\mathbf B}_\h}
\DeclareMathOperator{\eEbf}{\mathbf{e}^{\mathbf E}_\h}
\DeclareMathOperator{\ejbf}{\mathbf{e}^{\mathbf j}_\h}

\let\P\relax
\DeclareMathOperator{\P}{P}
\DeclareMathOperator{\eP}{e^{\P}_\h}
\DeclareMathOperator{\PI}{P_I}
\newcommand{\omegabold}{\boldsymbol \omega}

\newcommand{\vbf}{\mathbf v}
\newcommand{\Cbf}{\mathbf C}
\newcommand{\Fbf}{\mathbf F}
\newcommand{\kbf}{\mathbf k}
\newcommand{\Gbf}{\mathbf G}
\DeclareMathOperator{\Q}{Q}
\newcommand{\mubold}{\boldsymbol \mu}
\newcommand{\ubfh}{\mathbf u_\h}
\newcommand{\Bbfh}{\mathbf B_\h}
\newcommand{\Ebfh}{\mathbf E_\h}

\newcommand{\jbfh}{\mathbf j_\h}
\newcommand{\Hbfh}{\mathbf H_\h}
\DeclareMathOperator{\Ph}{P_\h}
\DeclareMathOperator{\Qh}{Q_\h}
\newcommand{\omegaboldh}{\boldsymbol \omega_\h}
\newcommand{\vbfh}{\mathbf v_\h}
\newcommand{\Cbfh}{\mathbf C_\h}
\newcommand{\Fbfh}{\mathbf F_\h}
\newcommand{\kbfh}{\mathbf k_\h}
\newcommand{\Gbfh}{\mathbf G_\h}
\newcommand{\muboldh}{\boldsymbol \mu_\h}
\let\Re\relax
\DeclareMathOperator{\Re}{R_e}
\DeclareMathOperator{\Remo}{R_e^{-1}}
\DeclareMathOperator{\Rm}{R_m}
\DeclareMathOperator{\Rmmo}{R_m^{-1}}
\DeclareMathOperator{\cL}{c}
\newcommand{\zerobf}{\mathbf 0}
\newcommand{\Bbfz}{\Bbf^0}
\newcommand{\Bbfzh}{\Bbfz_\h}
\newcommand{\ubfz}{\ubf^0}
\newcommand{\ubfzh}{\ubfz_\h}
\DeclareMathOperator{\X}{X}
\DeclareMathOperator{\Xh}{\X_\h}
\DeclareMathOperator{\Xtildeh}{\widetilde\X_\h}
\DeclareMathOperator{\Xbarh}{\overline\X_\h}
\newcommand{\Abf}{\mathbf A}
\newcommand{\Abfh}{\mathbf A_\h}
\newcommand{\Phih}{\Phi_\h}
\newcommand{\Hzcurl}{H_0(\curl,\Omega)}
\newcommand{\Hzdiv}{H_0(\div,\Omega)}
\newcommand{\Hznabla}{H_0(\nabla,\Omega)}

\newcommand{\Hzcurlh}{H_0^\h(\curl,\Omega)}
\newcommand{\Hzbarcurlh}{\overline H_0^\h(\curl,\Omega)}
\newcommand{\Hzdivh}{H_0^\h(\div,\Omega)}
\newcommand{\Hztildediv}{\widetilde H_0(\div,\Omega)}
\newcommand{\Hztildedivh}{\widetilde H_0^\h(\div,\Omega)}
\newcommand{\Hznablah}{H_0^\h(\nabla,\Omega)}
\newcommand{\Hcurlh}{H^\h(\curl,\Omega)}
\newcommand{\Hdivh}{H^\h(\div,\Omega)}
\newcommand{\Hnablah}{H^\h(\nabla,\Omega)}
\DeclareMathOperator{\Hcalm}{\mathcal H_m}
\DeclareMathOperator{\Hcalc}{\mathcal H_c}
\DeclareMathOperator{\LHS}{LHS}
\DeclareMathOperator{\RHS}{RHS}
\DeclareMathOperator{\ds}{ds}
\DeclareMathOperator{\dr}{dr}
\newcommand{\mumo}{\mu^{-1}}
\newcommand{\cP}{c_P}
\newcommand{\Pbb}{\mathbb P}
\newcommand{\E}{K}

\newcommand{\F}{F}

\newcommand{\hE}{\h_\E}

\newcommand{\taun}{\mathcal T_\h}
\newcommand{\Fcaln}{\mathcal F_\h}
\newcommand{\Ecaln}{\mathcal E_\h}
\newcommand{\FcalE}{\mathcal F^\E}
\newcommand{\EcalE}{\mathcal E^\E}
\newcommand{\nbfE}{\mathbf n_\E}
\newcommand{\nbfF}{\mathbf n_\F}

\usepackage{color}
\title{\normalsize Error estimates for a helicity-preserving finite element
discretisation of an incompressible magnetohydrodynamics system}
\author{\normalsize{L. Beir\~ao~da~Veiga\thanks{Dipartimento di Matematica e Applicazioni,
    Universit\`a degli Studi di Milano-Bicocca, Italy (lourenco.beirao@unimib.it, lorenzo.mascotto@unimib.it)}
    \thanks{IMATI-CNR, 27100, Pavia, Italy},\;
K. Hu\thanks{School of Mathematics, University of Edinburgh, UK (kaibo.hu@ed.ac.uk)},\;
L. Mascotto\footnotemark[1]
\thanks{Faculty of Mathematics, University of Vienna, 1090 Vienna, Austria}
\footnotemark[2]}}
\normalsize
\date{}

\begin{document}   
\maketitle
\begin{abstract}
\noindent We derive error estimates of
a finite element method
for the approximation of solutions to a seven-fields formulation
of a magnetohydrodynamics model,
which preserves the energy of the system,
and the magnetic and cross helicities on the discrete level.\\
%%%%

%%%%

\end{abstract}

%%%
\section{Introduction} \label{section:introduction}
\paragraph*{State-of-the-art.}
The numerical discretisation of incompressible magnetohydrodynamics (MHD) systems
has drawn significant interest inspired by applications
in plasma physics and fusion energy research.
Finite element methods provide one of the approaches
for solving such coupled multiphysics problems.
Examples of early works can be found
in~\cite{Gunzburger-Meir-Peterson:1991,Schoetzau:2004}.
In recent years, various discussions on finite element methods
have been focused on constructing schemes
that precisely preserve certain quantities up to machine precision.
A method preserving the energy and the divergence-free condition of the magnetic field 
(magnetic Gauss law) was studied in~\cite{Hu-Ma-Xu:2017}.
Error estimates for such a method or its variants
can be derived by adapting the proofs in the framework
of the virtual element method~\cite{BeiraodaVeiga-Dassi-Manzini-Mascotto:2023},
or extending the proofs for stationary problems as in~\cite{Hu-Xu:2019,Hu-Qiu-Shi:2020}.
The method in~\cite{Hu-Ma-Xu:2017} uses the velocity~$\ubf$, the pressure~$p$,
the magnetic field~$\Bbf$, and the electric field~$\Ebf$
as main variables from a de~Rham complex.
We shall refer to this formulation as the {\it four-field scheme}.
Another approach based on the magnetic potential, as well as its convergence analysis,
can be found in~\cite{Hiptmair-Li-Mao-Zheng:2018}.
The convergence analysis for finite element discretisations of incompressible MHD systems
was also carried out in~\cite{Gao-Qiu-Sun:2023}. 
Another important aspect in the
derivation of error estimates of MHD systems
is the robustness with respect to the physical parameters,
an aspect which is beyond the scopes of the present contribution;
among others, we recall the related
contributions~\cite{BeiraodaVeiga-Dassi-Vacca:2024,Gerbeau-LeBris-Lelievre:2006}
for the stationary linearised case
and~\cite{BeiraodaVeiga-Dassi-Vacca:2024B}
for the fully nonlinear dynamic equations.

The helicity of divergence-free fields is a quantity encoding the topology
of the fields~\cite{Arnold-Keshin:2009,Moffatt:1969}.
In fluids, the fluid helicity $\int \ubf\cdot \bm w\, dx$ characterises the knots of the vorticity.
In MHD systems, two kinds of helicity exist:
the magnetic helicity $\int \Bbf\cdot \Abf\, dx$ characterises knots of the magnetic field,
where~$\Abf$ is the magnetic potential satisfying $\nabla\times \Abf=\Bbf$;
the cross helicity  $\int \Bbf\cdot \ubf\, dx$ describes knots between the vorticity
and the magnetic fields.

The helicity has a fundamental importance in various aspects of fluid mechanics and MHD,
such as turbulence~\cite{Moffatt:1981} and magnetic relaxation~\cite{Pontin-Hornig:2020},
and is conserved in ideal flows.
More precisely, the magnetic helicity is conserved as long as the magnetic diffusion vanishes,
the magnetic Reynolds number being infinity.
The general philosophy of structure-preserving and compatible discretisation
suggests that preserving helicity is important for physical fidelity.
In many important examples, there are indeed concrete reasons.
For example, a fundamental question in plasma physics is how the system evolves
with given initial data, which is related to open questions
existing today such as Parker's hypothesis~\cite{Pontin-Hornig:2020}.
Topological barriers, as encoded in helicity, constrain the behaviour
of the magnetic field under the relaxation.
Establishing a corresponding mechanism on the discrete level
is important for correctly computing the dynamical behaviours
of the plasma~\cite{He-Hu-Farrell:preparation}.
Along this direction, the work~\cite{Hu-Lee-Xu:2021}
considered a seven-fields finite element scheme
that preserves the energy, the magnetic Gauss law,
and the magnetic and cross helicities at once:
additional mixed variables, i.e.,
the magnetic field $\Hbf$, the vorticity $\omegabold$,
and the current density $\jbf$, which are actually
natural physical variables in the original system,
were added to the variables in the four-field scheme~\cite{Hu-Ma-Xu:2017}.
A key idea in that construction was to use $L^{2}$ projections
into proper spaces from a de~Rham sequence.
This approach was also explored by Rebholz to derive a scheme
that preserves the fluid helicity for the Navier-Stokes equations~\cite{Rebholz:2007}.
The construction was extended to the Hall MHD system~\cite{Laakmann-Hu-Farrell:2023}.
Other recent works on preserving the helicity for the MHD or the Navier-Stokes equations
can be found in~\cite{Gawlik-GayBalmaz:2022,Zhang-Palha-Gerritsma-Rebholz:2022}. 

Nevertheless, it is also important to understand
the {\it behaviour and limit of structure-preserving schemes}.
The $L^{2}$ projections played an important role
in the construction of the method proposed in~\cite{Hu-Lee-Xu:2021}.
Although convergence was observed in the numerical tests~\cite{Hu-Lee-Xu:2021},
the convergence in general situations with various norms is not completely clear.
In fact, it is not difficult to imagine that structure-preserving methods must have a limit.
For example, for the Navier-Stokes equation,
the Onsager conjecture~\cite{Onsager:1949,Constantin-Weinan-Titi:1994,Isett:2018}
claims that under certain conditions,
the energy conservation does not hold on the continuous level
technically because the regularity required by the integration by parts
is not valid for rough solutions.
The issue concerned by the Onsager conjecture is closely related
to the mathematical properties of the Navier-Stokes equations and turbulence,
and is thus of both mathematical and physical importance.
The original Onsager conjecture is concerned with the energy conservation of fluids.
However, various versions of theorems and conjectures in the same spirit
exist for the helicity preservation for both
the Navier-Stokes and the MHD equations~\cite{Shvydkoy:2010}.
Despite the fact that quantities may not be conserved on the continuous level,
most finite element methods for the Navier-Stokes equations
preserve the energy by construction.
Therefore, they cannot be used to compute certain classes of rough solutions,
and structure-preserving methods might play the opposite role
by producing spurious solutions in such scenarios.
However, to the best of our knowledge, this issue was not
extensively discussed in the literature, with a few exceptions,
see, e.g., \cite{Fehn-Kronbichler-Munch-Wall:2022}. 
The above motivation indicates that investigating
convergence issues of energy-
and helicity-preserving schemes is a critical aspect for reliable MHD computations.

\paragraph*{Contributions of this paper.}
Here,
we provide the first error estimates
for the recent families of energy-helicity-preserving finite element schemes
for incompressible MHD equations~\cite{Hu-Lee-Xu:2021}.
The analysis takes inspiration
from that in the virtual element method framework~\cite{BeiraodaVeiga-Dassi-Manzini-Mascotto:2023},
here combined with the specific challenges of the $7$-field formulation.
% Although rough solutions serve as a motivation,
% we focus on reasonably smooth solutions on the continuous level.
We focus on error estimates for 
the semi-discrete in space method.
The fully discrete scheme in~\cite{Hu-Lee-Xu:2021}
used a mid-point rule discretization in time,
which is only an example of a larger family of time integrators
preserving quadratic invariants.
The extension of the error bounds
proposed in this work to the fully discrete case
may carry major or minor challenges
depending on the particular choice of the time integrator.

\paragraph*{Notation.}
Given~$v:\Rbb^3\to\Rbb$ and~$\vbf : \Rbb^3 \to \Rbb^3$, we define
\[
\begin{split}
\nabla v := (\partial_x v, \partial_y v, \partial_z v)^T,
\qquad\qquad
\div \vbf := \partial_x v_z + \partial_y v_2 + \partial_z v_3,\\
\curl \vbf :=  (\partial_y v_3 - \partial_z v_2,
                \partial_z v_1 - \partial_x v_1,
                \partial_x v_2 - \partial_y v_1)^T.
\end{split}
\]
Given~$D$ a Lipschitz domain in~$\Rbb^3$
with diameter~$\h_D$,
we introduce~$H^s(\nabla,D)$ as the usual Sobolev space of positive order~$s$.
For~$s=1$, we omit the Sobolev order and write~$H(\nabla,D)$.
The Sobolev space of order~$s=0$ is the usual Lebesgue space~$L^2(D)$;
its subspace of functions with zero average over~$D$ is~$L^2_0(D)$.

We denote the Sobolev inner product, seminorm, and norm by
\[
(\cdot,\cdot)_{s,D},\qquad\qquad\qquad
\SemiNorm{\cdot}{s,D},\qquad\qquad\qquad
\Norm{\cdot}{s,D}.
\]
Henceforth, whenever clear, we shall omit the dependence on the domain~$D$.
We shall further omit the Sobolev index~$s$ when~$s=0$.

On the boundary~$\partial D$ of~$D$,
we define the space~$H^{\frac12}(\partial D)$
as the image of~$H(\nabla,D)$ through the standard trace operator.
The dual space of~$H^{\frac12}(\partial D)$ is given by~$H^{-\frac12}(\partial D)$.

The spaces of $L^2(D)$ functions with curl and divergence in~$L^2(D)$ are~$H(\curl,D)$ and~$H(\div,D)$,
which we endow with the norms
\[
\Norm{\cdot}{\curl,D}^2 
:= \h_D^{-2} \Norm{\cdot}{0,D}^2 + \Norm{\curl(\cdot)}{0,D}^2,
\qquad\qquad
\Norm{\cdot}{\div,D}^2 
:= \h_D^{-2} \Norm{\cdot}{0,D}^2 + \Norm{\div(\cdot)}{0,D}^2.
\]
For any positive~$s$, $H^s(\curl,D)$ and $H^s(\div,D)$ denote
the subspaces of~$H(\curl,D)$ and~$H(\div,D)$ of~$H^s(D)$ functions
with~$\curl$ and~$\div$ in~$H^s(D)$.

There exist trace operators from~$H(\curl,D)$ and~$H(\div,D)$
into $H^{-\frac12}(\partial D)$; see, e.g., \cite{Monk-2003, Buffa-Ciarlet:2001}.
We introduce~$H_0(\nabla,D)$, $H_0(\curl,D)$, and~$H_0(\div,D)$
as the subspaces of functions in~$H(\nabla,D)$, $H(\curl,D)$, and~$H(\div,D)$
with zero standard, tangential, and normal traces
in~$H^{\frac12}(\partial D)$, $H^{-\frac12}(\partial D)$,
and~$H^{-\frac12}(\partial D)$, respectively.

For a given positive~$T$, a nonnegative $\ell$,
an integer~$s$, and a Hilbert space~$H$,
we shall also use the Bochner spaces~$H^\ell(0,T; H)$
and~$W^{s,\infty}(0,T;H)$.
The latter is a Sobolev space with finite norm
\[
\Norm{v}{W^{s,\infty}(0,T;H)}
:= \text{essSup}_{t\in(0,T)} \Norm{\partial_s v(\cdot,t)}{H}.
\]

\paragraph*{Meshes.}
In what follows, $\{\taun\}$ denotes a sequence of conforming tetrahedral tessellation
of a given polyhedral domain.
We assume that~$\{\taun\}$ is uniformly shape-regular
with shape-regularity parameter~$\sigma$.
The sets of faces and edges of~$\taun$ are~$\Fcaln$ and~$\Ecaln$.
Given~$\E$ in~$\taun$, $\FcalE$ and~$\EcalE$ are its sets of faces and edges;
$\nbfE$ is its outward unit normal vector;
$\hE$ is its diameter.
The maximum of all~$\hE$ is denoted by~$\h$.
With each face~$\F$ in~$\Fcaln$, we fix once and for all~$\nbfF$
as one of the two unit normal vectors;
if~$\F$ belongs also to~$\FcalE$ for a given element~$\E$,
we have~$\nbfE{}_{|\F} = \pm \nbfF$.
The space of piecewise polynomials of maximum degree~$\ell$
over~$\taun$ is denoted by~$\Pbb_k(\taun)$.

The forthcoming analysis also works on tensor product meshes
with minor modifications;
we stick to the simplicial case for the presentation's sake.

\paragraph*{Outline.}
We present the seven-fields formulation of the MHD model in Section~\ref{section:model-problem};
there, we also describe certain quantities (energy, magnetic and cross helicities) that are preserved.
In Section~\ref{section:method},
we recall the FEM from~\cite{Hu-Lee-Xu:2021}
for the approximation of solutions to the MHD model under consideration;
we further show that the energy, and the magnetic and cross helicities
are preserved on the discrete level.
We exhibit optimal a priori estimates for the semi-discrete scheme
under suitable regularity assumptions on the exact solution
to the continuous problem in Section~\ref{section:convergence-semi-discrete}.
Appendix~\ref{appendix:helicity} is concerned with recalling the proof
of the discrete helicity-preservation.

%%%
\section{The model problem} \label{section:model-problem}
%%%
Let~$\Omega$ be a contractible, Lipschitz polyhedral domain in~$\Rbb^3$ with boundary~$\Gamma$;
$\nbfGamma$ the unit outward unit vector to~$\Gamma$;
$\Re$ and~$\Rm$ the fluid and magnetic Reynolds numbers;
$\cL$ the coupling number given by the ratio of the Alven and fluid speeds;
$\mu$ the permeability of the medium in~$\Omega$.

Consider the following MHD system of equations in~$\Omega\times(0,T]$:
find $(\ubf, \omegabold, \jbf, \Ebf, \Hbf, \Bbf, \P)$ such that
\begin{subequations} \label{MHD-strong}
\begin{align}
\dt \ubf - \ubf\times\omegabold + \Remo \curl\curl\ubf - \cL \jbf\times\Bbf +\nabla\P &= \fbf \label{MHD-strong-a}\\
\jbf -  \curl \Hbf &= \zerobf \label{MHD-strong-b}\\
\dt \Bbf + \curl \Ebf &= \zerobf \label{MHD-strong-c}\\
\Rmmo \jbf - (\Ebf + \ubf \times \Bbf) &= \zerobf \label{MHD-strong-d}\\
\div \ubf &= \zerobf \label{MHD-strong-e}\\
\omegabold - \curl \ubf &= \zerobf \label{MHD-strong-f}\\
\Bbf -\mu\Hbf &= \zerobf\label{MHD-strong-g}.
\end{align}
\end{subequations}
In words, we look for a current density~$\jbf$;
electric and magnetic fields~$\Ebf$ and~$\Bbf$
that induce the Lorentz force~$\cL (\jbf \times \Bbf)$ acting on a conductive fluid (plasma)
with pointwise velocity field~$\ubf$ and scalar total pressure~P;
a vorticity of the plasma~$\omegabold$;
a magnetic induction~$\Hbf$.
Given~$p$ the pointwise pressure of the plasma,
we define the total pressure as
\[
\P := p + \vert \ubf \vert^2.
\]
The first five equations in~\eqref{MHD-strong} are standard in MHD formulations:
the first one is a fluid momentum balance
equation
where the last term on the left-hand side represents
the Lorentz force generated by the electromagnetic fields;
the second is Amp\'ere's circuital law;
the third is Faraday's law;
the fourth is Ohm's law neglecting the contributions of the Coulomb force;
the fifth represents the mass conservation of the plasma,
i.e., the incompressibility of the plasma.
In particular, model~\eqref{MHD-strong} is the MHD system
in vorticity formulation with explicit constitutive laws.

We endow the above system with the initial conditions
\begin{equation} \label{initial-conditions}
\ubfz(\cdot):= \ubf(\cdot,0),
\qquad\qquad\qquad
\Bbfz(\cdot):= \Bbf(\cdot,0)
\quad\text{such that}\quad \div \Bbfz(\cdot)=0
\quad \text{in } \Omega,
\end{equation}
and the boundary conditions on~$\Gamma$
\begin{equation} \label{boundary-conditions}
\ubf\times\nbfGamma = \zerobf,
\qquad\qquad
\P = 0,
\qquad\qquad
\Bbf \cdot \nbfGamma = \zerobf ,
\qquad\qquad
\Ebf \times \nbfGamma = \zerobf .
\end{equation}
Equation~\eqref{MHD-strong-c} and the fact that~$\Bbfz$
is divergence free in~$\Omega$, see~\eqref{initial-conditions},
also imply that~$\Bbf$ is divergence free for all times.

The third and fourth conditions in~\eqref{boundary-conditions}
are ``physical'' boundary conditions for the magnetic and electric fields;
instead, the first and second ones are instrumental
for the formulation of the motion fluid equation~\eqref{MHD-strong-a}
in Lamb form~\cite{Lamb:1924}.

% Given
% \[
% \Hztildediv
% := \{ \Cbf \in \Hzdiv \mid \div\Cbf=0 \},
% \]
We introduce the space
\begin{equation} \label{space:X}
\X := [\Hzcurl]^5 \times \Hzdiv \times \Hznabla
\end{equation}
and the bilinear form
\[
a(\ubf,\vbf) := (\curl \ubf, \curl \vbf)
\qquad\qquad\qquad
\forall \ubf,\vbf \in \Hzcurl.
\]
A weak formulation of~\eqref{MHD-strong} reads:
find $(\ubf, \omegabold, \jbf, \Ebf, \Hbf, \Bbf, \P)$ in~$\X$ such that
\begin{subequations} \label{MHD-weak}
\begin{align}
(\dt \ubf,\vbf) - (\ubf\times\omegabold,\vbf) + \Remo a(\ubf,\vbf) - \cL (\jbf\times\Bbf,\vbf) + (\nabla\P,\vbf)                                       &= (\fbf,\vbf) \label{MHD-weak-a} \\
\mu(\jbf,\kbf) - (\Bbf,\curl\kbf)              &= 0 \label{MHD-weak-b} \\
(\dt \Bbf,\Cbf) + (\curl \Ebf,\Cbf)            &= 0 \label{MHD-weak-c} \\
(\Rmmo \jbf - [\Ebf + \ubf \times \Bbf], \Gbf) &= 0 \label{MHD-weak-d} \\
(\ubf,\nabla \Q)                               &= 0 \label{MHD-weak-e} \\
(\omegabold,\mubold) - (\ubf,\curl\mubold)     &= 0 \label{MHD-weak-f} \\
(\Bbf,\Fbf) -\mu(\Hbf,\Fbf)                    &= 0 \label{MHD-weak-g}
\end{align}
\end{subequations}
for all $(\vbf,\mubold, \kbf, \Fbf, \Gbf, \Cbf, \Q)$ in~$\X$.
Equation~\eqref{MHD-weak-b} is derived
from equations~\eqref{MHD-strong-b} and~\eqref{MHD-strong-g}.

An existence result is given in~\cite[Proposition 2.19]{Gerbeau-LeBris-Lelievre:2006}
for specific choices of the initial and boundary conditions.
A uniqueness result can be found in~\cite[Section 2.2.2.4]{Gerbeau-LeBris-Lelievre:2006}
for small times under suitable assumptions on the data.

%%%
\subsection{Preservation of the energy, and the magnetic and cross helicities} \label{subsection:helicities}
%%%

%%
\paragraph*{Preservation of the energy.}
The MHD system~\eqref{MHD-weak} preserves the energy
as long as homogeneous boundary conditions as in~\eqref{boundary-conditions} are imposed;
see, e.g., \cite[Theorem~1]{Hu-Lee-Xu:2021} for a proof.
\begin{thm} \label{theorem:energy-preservation-continuous}
The following identity holds true:
\[
\frac12\dt \Norm{\ubf}{}^2
+ \frac{\cL\mumo}{2} \dt\Norm{\Bbf}{}^2
+ \Remo \Norm{\curl \ubf}{}^2
+ \cL \Rmmo \Norm{\jbf}{}^2
= (\fbf,\ubf).
\]
\end{thm}

\paragraph*{Preservation of the magnetic and cross helicities.}
Given the magnetic field~$\Bbf$, let~$\Abf$ be an associated magnetic potential, i.e.,
\[
\curl\Abf=\Bbf.
\]
Given also the velocity field~$\ubf$,
we define the magnetic and cross helicities as
\[
\Hcalm:= (\Abf,\Bbf),
\qquad\qquad\qquad
\Hcalc:= (\ubf, \Bbf).
\]
We have the following noteworthy properties of the MHD system~\eqref{MHD-weak};
see , e.g., \cite[Lemma~1]{Hu-Lee-Xu:2021} for a proof.

\begin{thm} \label{theorem:helicity-preservation-continuous}
The following identities hold true:
\small{\[
\begin{split}
& \dt \Hcalm = -((\dt\Abf+2\Ebf)\times\Abf, \nbfGamma )_{0,\Gamma}
                - 2\Rmmo\mumo (\Bbf, \curl \Bbf) ,\\
& \dt \Hcalc = ( [\ubf\times\Bbf]\times\ubf - \P \Bbf -\Rmmo (\curl\Bbf)\times\ubf \Remo \omegabold\times\Bbf, \nbfGamma)_{0,\Gamma} 
    + (\fbf, \Bbf)
    - (\Remo+\Rmmo\mumo) (\curl\Bbf,\curl\ubf).
\end{split}
\]}\normalsize{}
\end{thm}
\begin{rem} \label{remark:helicity-continuous}
The right-hand sides in the identities of Theorem~\ref{theorem:helicity-preservation-continuous} above vanish if:
\begin{itemize}
    \item suitable homogeneous boundary conditions are imposed;
    \item there is no source term~$\fbf$ in the fluid motion equation;
    \item we consider the ideal MHD system, i.e.,
    \eqref{MHD-weak} imposing (formally) $\Re=\Rm=\infty$.
\end{itemize}
A formulation with constant in time magnetic and cross helicities
is called \emph{helicity-preserving}.
\end{rem}

%%%
\section{The method} \label{section:method}
%%%
We introduce a finite element method
for the approximation of solutions to~\eqref{MHD-weak},
including a careful description of the discrete version
of the space~$\X$ in~\eqref{space:X},
and discuss some important properties,
including the energy, and the magnetic and cross helicities preservation
on the discrete level.

%%%
\subsection{Discrete spaces, discrete operators and various approximants} 
\label{subsection:discrete-spaces}
%%%

In this section, we review known results from the literature and derive additional technical estimates,
which will be needed in the sequel.
We consider finite element spaces~$\Hznablah$, $\Hzcurlh$, and~$\Hzdivh$
that retain the conformity of~$\Hznabla$, $\Hzcurl$, and~$\Hzdiv$, respectively.
Notably, for~$k$ in $\Nbb_0$,
we take the usual Lagrange, first kind N\'ed\'elec, and Raviart-Thomas elements
of order~$k+1$, $k$, and~$k$, respectively, over a given mesh~$\taun$.
On each element~$\E$, these spaces are endowed with the following degrees of freedom:
\begin{itemize}
    \item (\emph{Lagrange elements}) the point values at 
    equally distributed Lagrangian points fixing a polynomial of degree~$k+1$;
    such degrees of freedom are well defined for functions in~$H^{\frac32+\varepsilon}(\E)$ for any arbitrarily small and positive~$\varepsilon$;
    \item (\emph{N\'ed\'elec elements}) (scalar) edge moments of the tangential components up to order~$k$,
    face moments of the (2-vector) tangential components up to order~$k-1$,
    and (vector) elemental moments up to order~$k-2$;
    such degrees of freedom are well defined
    (up to an edge-to-cell lifting of the degrees of freedom
    \cite[Sect. 17.3]{Ern-Guermond:2021A})
    for vector fields
    in~$[H^{\frac12+\delta}(\E)]^3$ ($\delta$ arbitrarily small and positive)
    such that their $\curl$ belongs to~$[L^q(\E)]^3$
    ($q$ larger than~$2$), see~\cite[eq. (7)]{Boffi-Gastaldi:2006};
    \item (\emph{Raviart-Thomas elements}) (scalar) face moments of the normal components up to order~$k$
    and (vector) elemental moments up to order~$k-1$;
    such degrees of freedom are well defined
    (up to a face-to-cell lifting of the degrees of freedom
    \cite[Sects. 17.1 and 17.2]{Ern-Guermond:2021A})
    for vector fields in~$H(\div,\E)\cap [L^q(\E)]^3$ ($q$ larger than~$2$),
    see \cite[Section~$2.5.1$]{Boffi-Brezzi-Fortin:2013}.
\end{itemize}
The corresponding global spaces are constructed
by $H^1$, $H(\curl)$, and $H(\div)$ conforming coupling
of their local counterparts~\cite{Ern-Guermond:2021A}.
Suitable zero traces over~$\partial\Omega$ are naturally enforced in the above spaces.
The corresponding spaces with free traces
are denoted by~$\Hcurlh$, $\Hdivh$, and~$\Hnablah$, respectively.

\paragraph*{Exact sequences.}
In what follows, we shall use the following continuous
\[
\Hznabla
\quad\overset{\nabla}{\longrightarrow}\quad
\Hzcurl
\quad\overset{\curl}{\longrightarrow}\quad
\Hzdiv
\quad\overset{\div}{\longrightarrow}\quad
L^2_0(\Omega)
\quad\overset{0}{\longrightarrow}\quad
0
\]
and discrete
\begin{equation} \label{discrete-exact-sequence}
\Hznablah
\quad\overset{\nabla}{\longrightarrow}\quad
\Hzcurlh
\quad\overset{\curl}{\longrightarrow}\quad
\Hzdivh
\quad\overset{\div}{\longrightarrow}\quad
\Pbb_{k}(\taun)\setminus\Rbb
\quad\overset{0}{\longrightarrow}\quad
0
\end{equation}
exact sequence structures
(recall that $\Omega$ is contractible).

\paragraph*{An interpolation-type operator in~$H(\nabla)$.}
The Lagrangian interpolant of order~$k+1$ of~$\P$ in $\Hznabla\cap H^s(\nabla,\Omega)$, $s>3/2$,
is defined as the unique function~$\PI$ in~$\Hznablah$ satisfying
\begin{equation} \label{definition:Lagrangian-interpolant}
(\P-\PI)(\nu) = 0 
\qquad\qquad
\forall \nu \text{ Lagrangian nodes of order } k+1.
\end{equation}
Under extra regularity assumptions, local interpolation estimates are standard;
see, e.g., \cite{Ciarlet:2002}.
\begin{lem} \label{lemma:Lagrangian-interpolant}
Let~$\E$ be an element of a regular simplicial tessellation~$\taun$ of~$\Omega$.
Given~$\P$ in $H^{k+2}(\nabla,\E) \cap H_0(\nabla,\E)$,
let~$\PI$ be its Lagrange interpolant.
Then, there exists a positive constant~$C$ independent of~$\hE$ and~$\P$
but dependent on the shape-regularity parameter~$\sigma$ of the mesh
and the polynomial degree~$k$ such that
\[
\Norm{\nabla(\P-\PI)}{0,\E}
\le C \hE^{k+1} \SemiNorm{\P}{k+2,\E},
\qquad\qquad
\Norm{\P-\PI}{0,\E}
\le C \hE^{k+2} \SemiNorm{\P}{k+2,\E}.
\]
\end{lem}

\paragraph*{Interpolation and commuting operators in~$H(\curl)$.}
First, we introduce the N\'ed\'elec interpolant~$\EbfI$
of a field~$\Ebf$ in $[H^{\frac12+\delta}(\Omega)]^3$
($\delta$ positive arbitrarily small)
with $\curl \Ebf$ in $[L^q(\Omega)]^3$, $q$ larger than~2.
We have standard interpolation estimate results
\cite[Theorem~5.41 and Remark~5.42]{Monk-2003}.
We also include $L^\infty$ estimates,
which can be derived from the $L^2$ bounds by standard arguments,
exploiting the local nature of the interpolation operator
and its polynomial preservation property;
they are needed in the proof of Proposition~\ref{proposition:L-infty-stab-PiNcal} below.
Interpolation estimates under lower regularity of~$\Ebf$
are also available in the literature~\cite{Boffi-Gastaldi:2006, Ern-Guermond:2021A}.
We prefer sticking to the current setting to keep
the presentation as simple as possible.

\begin{lem} \label{lemma:Monk-Nedelec}
Let~$\E$ be an element of a regular simplicial tessellation~$\taun$ of~$\Omega$.
Given~$\Ebf$ sufficiently smooth,
let~$\EbfI$ be its N\'ed\'elec interpolant.
Then, there exists a positive constant~$C$ independent of~$\hE$ and~$\Ebf$,
but dependent on the shape-regularity parameter~$\sigma$ of the mesh
and the polynomial degree~$k\ge1$ such that
\begin{equation} \label{Nedelec-interpolation-estimates}
\Norm{\Ebf-\EbfI}{0,\E}
\le C \hE^{k+1} \SemiNorm{\Ebf}{k+1,\E},
\qquad\qquad
\Norm{\curl(\Ebf-\EbfI)}{0,\E}
\le C \hE^{k+1} \SemiNorm{\curl\Ebf}{k+1,\E} .
\end{equation}
\small{\begin{equation} \label{Nedelec-interpolation-estimates-infty}
\Norm{\Ebf-\EbfI}{L^\infty(\E)}
\le C \hE^{k+1} \SemiNorm{\Ebf}{W^{k+1,\infty}(\E)},
\qquad
\Norm{\curl(\Ebf-\EbfI)}{L^\infty(\E)}
\le C \hE^{k+1} \SemiNorm{\curl\Ebf}{W^{k+1,\infty}(\E)}.
\end{equation}}\normalsize
Estimates~\eqref{Nedelec-interpolation-estimates}
and~\eqref{Nedelec-interpolation-estimates-infty}
are also valid for the case~$k=0$ except for the first one,
which in that case reads \cite[eq. (16.17)]{Ern-Guermond:2021A}
\[
\Norm{\Ebf-\EbfI}{0,\E}
\le C \left(\hE \SemiNorm{\Ebf}{1,\E}
            + \hE^2 \SemiNorm{\Ebf}{2,\E} \right).
\]
\end{lem}
Additionally, we introduce a \emph{global} operator~$\PiNcal$
mapping~$\Hzcurl$ into~$\Hzcurlh$ as follows:
for all~$\Ebf$ in~$\Hzcurl$,
$\PiNcalE$ is the unique function in~$\Hzcurlh$ such that
\begin{equation} \label{interpolation:curl}
\begin{cases}
a(\Ebf-\PiNcalE, \Fbfh)   = 0   & \forall \Fbfh \in \Hzcurlh \\
(\Ebf-\PiNcalE,\nabla\Qh) = 0   & \forall \Qh \in \Hznablah.
\end{cases}
\end{equation}

The operator~$\PiNcalE$ is well-posed
and possesses certain approximation properties
as detailed in the next result.
\begin{prop} \label{proposition:PiNcalEbf-estimates}
For all~$\Ebf$ in~$H^{k+1}(\curl,\Omega)$,
let~$\EbfI$ be its N\'ed\'elec interpolant
as in Lemma~\ref{lemma:Monk-Nedelec}.
Then, there exist positive constants~$C_1$ and~$C_2$ independent of~$\hE$ and~$\Ebf$
but dependent on the shape-regularity parameter~$\sigma$ of the mesh
and the polynomial degree~$k$ such that
\begin{equation} \label{PiNcalEbf-estimates}
\Norm{\Ebf-\PiNcalE}{\curl,\Omega}
\le C_1 \Norm{\Ebf-\EbfI}{\curl,\Omega}
\le C_2 \h^{k+1} \SemiNorm{\Ebf}{H^{k+1}(\curl,\Omega)}.
\end{equation}
\end{prop}
\begin{proof}
Problem~\eqref{interpolation:curl} can be rewritten as follows:
find~$\PiNcalE$ and~$\Ph$ in $\Hzcurlh \times \Hznablah$ such that
\begin{equation} \label{problem:Ncal}
\begin{cases}
a(\PiNcalE, \Fbfh) + (\nabla\Ph,\Fbfh) = a(\Ebf, \Fbfh)     & \forall \Fbfh \in \Hzcurlh \\
(\PiNcalE,\nabla\Qh)                   = (\Ebf,\nabla\Qh)   & \forall \Qh \in \Hznablah .
\end{cases}
\end{equation}
Indeed, testing with $\Fbfh = \nabla\Ph$, we obtain~$\Ph=0$.

The bilinear form~$a(\cdot,\cdot)$ is coercive over the space of functions
in~$\Hzcurlh$ with zero divergence;
see~\cite[Corollary~3.51]{Monk-2003}.
Due to the discrete exact sequence structure~\eqref{discrete-exact-sequence},
the bilinear form~$(\nabla\cdot,\cdot)$
is inf-sup stable over
\[
\Hznablah \times \nabla\Hznablah
\subset
\Hznablah \times \Hzcurlh.
\]
The two bilinear forms and the two functionals on the right-hand side
of~\eqref{problem:Ncal} are continuous.
Therefore,
the standard inf-sup theory for mixed problems~\cite{Boffi-Brezzi-Fortin:2013} implies
that the above problem is well-posed with continuous dependence on~$\Ebf$.
Moreover, given~$\EbfI$ the N\'ed\'elec interpolant
of~$\PiNcalE$ as in Lemma~\ref{lemma:Monk-Nedelec},
there exists a positive constant~$c$
depending on the shape of~$\Omega$ such that
\[
\Norm{\PiNcalE-\EbfI}{\curl,\Omega}
= \Norm{\PiNcalE-\EbfI}{\curl,\Omega} + \Norm{\nabla\Ph}{1,\Omega}
\le c \Norm{\Ebf-\EbfI}{\curl,\Omega}.
\]
The assertion follows using the triangle inequality
and the estimates in~\eqref{Nedelec-interpolation-estimates}.
\end{proof}
The estimates in~\eqref{Nedelec-interpolation-estimates} provide us
with a local bound,
whereas the approximation estimates~\eqref{PiNcalEbf-estimates}
only give a global bound.

In what follows, we shall need the following technical result
stating the stability of the operator~$\PiNcal$
in the $L^\infty$ norm, under additional regularity on~$\Ebf$.
\begin{prop} \label{proposition:L-infty-stab-PiNcal}
Let~$\Ebf$ be in~$H^{\frac32}(\curl,\Omega) \cap [L^\infty(\Omega)]^3$
and~$\PiNcalE$ be as in~\eqref{interpolation:curl}.
Then, there exists a positive constant~$C$ independent of~$\h$
but dependent on the shape-regularity parameter~$\sigma$ of~$\taun$
and the polynomial degree~$k$ such that
\begin{equation} \label{stability-Linfty-PiNcal}
\Norm{\PiNcalE}{L^{\infty}(\Omega)} 
\le C \left(\min_{\E\in\taun} \hE \right)^{-\frac32} \h^{\frac32}
        \Big( \Norm{\Ebf}{H^{\frac32}(\curl,\Omega)}  \Big)
    + C \Norm{\Ebf}{L^\infty(\Omega)}.
\end{equation}
\end{prop}
\begin{proof}
Let~$\EbfI$ be the N\'ed\'elec interpolant of~$\Ebf$
as in Lemma~\ref{lemma:Monk-Nedelec}.
A polynomial inverse estimate entails
\[
\begin{split}
\Norm{\PiNcalE}{L^{\infty}(\Omega)} 
& \le \Norm{\Ebf - \PiNcalE}{L^{\infty}(\Omega)} 
        + \Norm{\Ebf}{L^{\infty}(\Omega)} \\
&  \le \Norm{\Ebf - \EbfI}{L^{\infty}(\Omega)} 
        + \max_{\E\in\taun} \Norm{\EbfI - \PiNcalE}{L^{\infty}(\E)} 
        + \Norm{\Ebf}{L^{\infty}(\Omega)} \\
& \le \Norm{\Ebf - \EbfI}{L^{\infty}(\Omega)} 
        + c_{inv} \hE^{-\frac32} \Norm{\EbfI - \PiNcalE}{L^{2}(\E)} 
        + \Norm{\Ebf}{L^{\infty}(\Omega)} \\
& \le \Norm{\Ebf - \EbfI}{L^{\infty}(\Omega)} 
        + c_{inv} \hE^{-\frac32} 
                ( \Norm{\Ebf - \PiNcalE}{L^{2}(\E)}  
                    + \Norm{\Ebf-\EbfI}{L^{2}(\E)} )
        + \Norm{\Ebf}{L^{\infty}(\Omega)}.
\end{split}
\]
The assertion follows using~\eqref{Nedelec-interpolation-estimates},
\eqref{Nedelec-interpolation-estimates-infty},
and~\eqref{PiNcalEbf-estimates},
and the fact that
\[
\Norm{\EbfI}{L^\infty(\Omega)}
\overset{\eqref{Nedelec-interpolation-estimates-infty}}{\lesssim}
\Norm{\Ebf}{L^\infty(\Omega)}.
\]
\end{proof}
%%%%%

%%
\paragraph*{An approximation, commuting operator in~$H(\div)$.}
We introduce the space of~$L^2$ Raviart-Thomas functions with zero divergence:
\begin{equation} \label{div-tilde-space}
\Hztildedivh
:= \{ \Bbfh\in\Hzdivh \mid \div \Bbfh=0 \}
\end{equation}
and a projection operator mapping~$[L^2(\Omega)]^3$
into~$\Hztildedivh$ as follows:
\begin{equation} \label{projection:divergence-free}
(\Bbf-\BbfI,\Cbfh) = 0 
\qquad\qquad\qquad \forall \Cbfh \in \Hztildedivh.
\end{equation}
The operator~$\Pitildez$ satisfies a crucial commuting property with~$\PiNcal$ defined in~\eqref{def:PiNk};
see display~\eqref{commuting-operators} below.

Let~$\Ebf$ in~$\Hzcurl$.
Using~\eqref{projection:divergence-free} entails
\begin{equation} \label{def:PiNk}
(\Pitildez(\curl\Ebf), \Cbfh) = (\curl\Ebf, \Cbfh)
\qquad\qquad \forall \Ebf\in\Hzcurl,\ \Cbfh \in \Hztildedivh.
\end{equation}
We have the following approximation result,
which is an immediate consequence of~\cite[Theorem~5.25 and Remark~5.26]{Monk-2003},
and the fact that the Raviart-Thomas interpolation preserves the divergence free property.

\begin{prop} \label{proposition:Hdiv-approx}
For all~$\Bbf$ in~$H^{k+1}(\E)$ such that~$\div\Bbf=0$,
there exists a positive constant~$C$ independent of~$\hE$ and~$\Bbf$
but dependent on the shape-regularity parameter~$\sigma$ of the mesh
and the polynomial degree~$K$ such that
\begin{equation} \label{Hdiv-approx}
\Norm{\Bbf-\BbfI}{0,\E}
\le C \hE^{k+1} \SemiNorm{\Bbf}{k+1,\E}.
\end{equation}
\end{prop}
For a given~$\Ebf$ in~$\Hzcurl$,
the fact that~$\div(\curl\Fbfh)=\zerobf$ implies
\[
a(\PiNcalE,\Fbfh)
\overset{\eqref{interpolation:curl}}{=}
a(\Ebf,\Fbfh)
\overset{\eqref{projection:divergence-free}}{=}
(\Pitildez (\curl\Ebf),\curl\Fbfh)
\qquad\qquad\forall \Fbfh\in\Hzcurlh.
\]
Therefore, we have the commuting property
\begin{equation} \label{commuting-operators}
\Pitildez(\curl \Ebf)  = \curl \PiNcalE.
\end{equation}

\paragraph*{A discrete curl operator and an $L^2$ projection.}
Define the discrete curl operator~$\curlh : \Hzdivh \to \Hzcurlh$ as
\begin{equation} \label{discrete-curl}
(\curlh \Bbfh, \vbfh) = (\Bbfh, \curl \vbfh)
\qquad\qquad \forall \vbfh\in\Hzcurlh.
\end{equation}
Also define the operator~$\Qbbh: [L^2(\Omega)]^3 \to \Hzcurlh$ as
\begin{equation} \label{L2-projection-curl}
(\vbf-\Qbbh\vbf, \Ebfh  ) =0
\qquad\qquad \forall \Ebfh\in\Hzcurlh.
\end{equation}
The operator~$\Qbbh$ acts as the identity on~$\Hzcurlh$,
i.e., is a projector.
Using Lemma~\ref{lemma:Monk-Nedelec},
we deduce the following result.
\begin{prop} \label{proposition:approximation:L2-proj-Hcurl}
For all~$\vbf$ in~$H^{k+1}(\Omega)$,
there exists a positive~$C$ independent of~$\h$
but dependent on the shape-regularity parameter~$\sigma$ of the mesh
and the polynomial degree~$k$ such that
\begin{equation} \label{approximation:L2-proj-Hcurl}
\Norm{\vbf-\Qbbh \vbf}{0,\Omega}
\le C \h^{k+1}  \SemiNorm{\vbf}{k+1, \Omega} .
\end{equation}
\end{prop}
Fix~$p$ in~$[1,\infty]$.
Being~$\Qbbh$ an $L^2$ projection operator
and using that the mesh is shape-regular,
\cite[Theorem~1]{Douglas-Dupon-Wahlbin:1974} guarantees
the existence of a positive constant~$C$
only depending on the shape-regularity parameter~$\sigma$ of the mesh,
the polynomial degree~$k$, and the Lebesgue index~$p$ such that
\begin{equation} \label{stability:L2-in-Linfty}
\Norm{\Qbbh \vbf}{L^p(\Omega)}
\le C \Norm{\vbf}{L^p(\Omega)}
\qquad\qquad\qquad \forall \vbf\in [L^p(\Omega)]^3.
\end{equation}

\subsection{The discrete problem, preservation properties and well-posedness}
Consider the following discrete counterpart of the space~$\X$ in~\eqref{space:X}:
\[
\Xh := [\Hzcurlh]^5 \times \Hzdivh \times \Hznablah.
\]
We are now in a position to introduce the
semi-discrete formulation of~\eqref{MHD-weak} given by:\\
find $(\ubfh, \omegaboldh, \jbfh, \Ebfh, \Hbfh, \Bbfh, \Ph)$ in~$\Xh$ such that,
for all time~$t$ in $(0,T]$,
\small{
\begin{subequations} \label{method-semidiscrete}
\begin{align}
(\dt \ubfh,\vbfh) - (\ubfh\times\omegaboldh,\vbfh) + \Remo a(\ubfh,\vbfh) 
- \cL (\jbfh\times(\mu\Hbfh),\vbfh) + (\nabla\Ph,\vbfh)  &= (\fbf,\vbfh)  \label{method-semidiscrete-a} \\
\mu(\jbfh,\kbfh) - (\Bbfh,\curl\kbfh)                    &= 0 \label{method-semidiscrete-b} \\
(\dt \Bbfh,\Cbfh) + (\curl \Ebfh,\Cbfh)                  &= 0 \label{method-semidiscrete-c} \\
(\Rmmo \jbfh - [\Ebfh + \ubfh \times(\mu\Hbfh)], \Gbfh)  &= 0 \label{method-semidiscrete-d} \\
(\ubfh,\nabla \Qh)                                       & = 0 \label{method-semidiscrete-e} \\
(\omegaboldh,\muboldh) - (\ubfh,\curl\muboldh)           &= 0 \label{method-semidiscrete-f} \\
(\Bbfh,\Fbfh) -\mu(\Hbfh,\Fbfh)                          &= 0 \label{method-semidiscrete-g}
\end{align}
\end{subequations}}\normalsize
for all $(\vbfh,\muboldh, \kbfh, \Fbfh, \Gbfh, \Cbfh, \Qh)$ in~$\Xh$.

System~\eqref{method-semidiscrete} is endowed with discrete initial conditions
\[
\ubfzh = (\ubfz)_I,
\qquad\qquad\qquad
\Bbfzh = \Pitildez \Bbfz,
\]
i.e., we interpolate the continuous initial conditions in~$\Hzcurlh$
and~$\Hztildedivh$ in the sense of Lemma~\ref{lemma:Monk-Nedelec}.
We are tacitly assuming that~$\ubfz$ satisfies the regularity assumptions
detailed in Section~\ref{subsection:discrete-spaces}
and~$\Bbfz$ belongs to $[L^2(\Omega)]^3$.
A weaker regularity on~$\ubfz$ is possible
resorting to quasi-interpolation operators as in~\cite{Ern-Guermond:2017}.

\paragraph*{Relations and properties of discrete vector fields.}
From the semi-discrete formulation~\eqref{method-semidiscrete},
and the definition of~$\Qbbh$ and~$\curlh$
in~\eqref{L2-projection-curl} and~\eqref{discrete-curl},
we deduce the following identities:
\begin{subequations}\label{strong-identities}
\begin{align}
\Ebfh       &
\overset{\eqref{L2-projection-curl},\eqref{method-semidiscrete-d}}{=}
\Rmmo \jbfh  - \Qbbh(\ubfh\times(\mu\Hbfh)),  \label{strong-identities-a} \\
\omegaboldh &
\overset{\eqref{L2-projection-curl},\eqref{method-semidiscrete-f}}{=}
\Qbbh (\curl\ubfh),                            \label{strong-identities-b} \\
\mu\jbfh    &
\overset{\eqref{discrete-curl},\eqref{method-semidiscrete-b}}{=}
\curlh\Bbfh,                                   \label{strong-identities-c} \\
\mu\Hbfh    &
\overset{\eqref{L2-projection-curl},\eqref{method-semidiscrete-g}}{=}
\Qbbh \Bbfh.                                   \label{strong-identities-d}
\end{align}
\end{subequations}
Equation~\eqref{method-semidiscrete-c} and the properties of the exact sequences imply
\begin{equation} \label{relation-Bh-curlEh}
\dt\Bbfh  = \curl\Ebfh
\qquad\Longrightarrow\qquad
- \dt\div\Bbfh = 0.
\end{equation}
We deduce
\begin{equation} \label{discrete-magnetic-zero-divergence}
\div\Bbfzh
= \div \Pitildez \Bbfz
\overset{\eqref{div-tilde-space},\eqref{projection:divergence-free}}{=} 0
\qquad
\overset{\eqref{relation-Bh-curlEh}}{\Longrightarrow}
\qquad
\div\Bbfh=0 \qquad \forall t \in [0,T].
\end{equation}

\paragraph*{Discrete energy preservation.}
We recall the following Poincar\'e-type inequality,
see, e.g., \cite[Corollary~3.51]{Monk-2003}:
\begin{equation} \label{Poincare:curl}
\Norm{\vbf}{} \le \cP \Norm{\curl \vbf}{}
\qquad\qquad
\forall \vbf \in H_0(\curl,\Omega)
\quad\text{with}\quad \div\vbf=0.
\end{equation}
The semi-discrete method~\eqref{method-semidiscrete}
preserves the energy of the system.
The following result is the discrete counterpart
of Theorem~\ref{theorem:energy-preservation-continuous}
and was proven in~\cite[Theorem~4]{Hu-Lee-Xu:2021}
for a specific full discretisation of~\eqref{method-semidiscrete}.

\begin{thm} \label{theorem:energy-preservation-discrete}
The following identity holds true:
\begin{equation} \label{energy-discrete-1}
\frac12\dt \Norm{\ubfh}{}^2
+ \frac{\cL\mumo}{2} \dt\Norm{\Bbfh}{}^2
+\Remo \Norm{\curl \ubfh}{}^2
+ \cL \Rmmo \Norm{\jbfh}{}^2
= (\fbf,\ubfh).
\end{equation}
For all~$t$ in~$(0,T]$ and~$\cP$ as in~\eqref{Poincare:curl},
we also have the following upper bound:
\begin{equation} \label{energy-discrete-2}
\begin{split}
& \Norm{\ubfh(\cdot,t)}{}^2
+ \cL \mumo \Norm{\Bbfh(\cdot,t)}{}^2
+ \Remo \int_0^t \Norm{\curl \ubf (\cdot,s)}{}^2 \ds
+ \cL \Rmmo \int_0^t \Norm{\jbfh(\cdot,s)}{}^2 \ds \\
& \le \Norm{\ubfzh(\cdot,t)}{}^2
+ \cL \mumo \Norm{\Bbfzh(\cdot,t)}{}^2
+ \Re \cP \int_0^t \Norm{\fbf(\cdot,s)}{}^2 \ds .
\end{split}
\end{equation}
\end{thm}
\begin{proof}
We provide details of the proof for completeness
since method~\eqref{method-semidiscrete} contains
more physical parameters than that in~\cite{Hu-Lee-Xu:2021}.

Standard properties of the cross product imply
\begin{equation} \label{trilinear:vanishes}
(\ubf \times \vbf,\ubf) = 0
\qquad\qquad\qquad
\forall \ubf,\ \vbf \in [L^2(\Omega)]^3.
\end{equation}
Take $\vbfh=\ubfh$ in~\eqref{method-semidiscrete-a} and~$\Qh=\Ph$ in~\eqref{method-semidiscrete-f}.
With this choice, the trilinear term involving the vorticity vanishes due to~\eqref{trilinear:vanishes}.
Therefore, using~\eqref{strong-identities-c} and~\eqref{strong-identities-d},
we write
\begin{equation} \label{initial-estimate-energy-preservation-discrete}
(\dt \ubfh, \ubfh) 
+ \Remo \Norm{\curl\ubfh}{}^2
- \cL\mumo ((\curlh\Bbfh)\times\Qbbh \Bbfh,\ubfh)
= (\fbf,\ubfh).
\end{equation}
On the other hand, picking~$\Cbfh=\Bbfh$ in~\eqref{method-semidiscrete-c},
we arrive at
\[
\begin{split}
(\dt \Bbfh, \Bbfh)
& \overset{\eqref{method-semidiscrete-c}}{=} - (\curl\Ebfh,\Bbfh) 
  \overset{\eqref{strong-identities-a}}{=} -(\curl[\Rmmo\jbfh-\Qbbh(\ubfh\times\Qbbh\Bbfh)],\Bbfh)   \\
& \overset{\eqref{discrete-curl}}{=} -\Rmmo(\jbfh,\curlh\Bbfh) + (\ubfh\times\Qbbh\Bbfh,\curlh\Bbfh)    \\
& \overset{\eqref{strong-identities-c}}{=} -\Rmmo\mu \Norm{\jbfh}{}^2 + (\ubfh\times\Qbbh\Bbfh, \curlh\Bbfh)  .    
\end{split}
\]
Exploiting the cross product's properties, we deduce
\[
\cL\mumo (\dt \Bbfh,\Bbfh)
 + \cL \Rmmo \Norm{\jbfh}{}^2
 = \cL \mumo (\ubfh\times\Qbbh\Bbfh, \curlh \Bbfh) 
 = - \cL \mumo ((\curlh \Bbfh) \times \Qbbh\Bbfh , \ubfh).
\]
Inserting this identity in~\eqref{initial-estimate-energy-preservation-discrete} entails
\[
(\dt \ubfh, \ubfh) 
+ \Remo \Norm{\curl\ubfh}{}^2
+ \cL\mumo (\dt \Bbfh,\Bbfh) 
+ \cL \Rmmo \Norm{\jbfh}{}^2
= (\fbf,\ubfh).
\]
Inequality~\eqref{energy-discrete-1} follows.
\medskip

Using~\eqref{MHD-strong-e} and~\eqref{Poincare:curl}, we deduce
\[
\begin{split}
\vert (\fbf,\ubfh) \vert
& \le \Norm{\fbf}{} \cP \Norm{\curl \ubf}{}
  \le \frac{\Re\cP}{2} \Norm{\fbf}{}^2 + \frac{\Remo}{2} \Norm{\curl \ubf}{}^2.
\end{split}
\]
Inequality~\eqref{energy-discrete-2} follows inserting
the above bound in~\eqref{energy-discrete-1} and performing standard manipulations.
\end{proof}

\paragraph*{Discrete magnetic and cross helicities-preservation.}
The finite element formulation~\eqref{method-semidiscrete}
is helicity-preserving in the ideal case.
In fact, the following result was proven in \cite[Theorems~5 and~6]{Hu-Lee-Xu:2021}
and is the discrete counterpart of Theorem~\ref{theorem:helicity-preservation-continuous}.
We report here below the result explicitly and review
also its proof in Appendix~\ref{appendix:helicity}
for completeness and since method~\eqref{method-semidiscrete} contains
more physical parameters than that in~\cite{Hu-Lee-Xu:2021}.

\begin{thm} \label{theorem:helicity-discrete}
Let~$\Abfh$ be any potential of the discrete magnetic field~$\Bbfh$, i.e.,
\[
\Abfh \in \Hzcurlh
\qquad\qquad \text{ be such that} \qquad\qquad
\curl \Abfh = \Bbfh.
\]
Then, the two following identities involving the discrete magnetic and cross helicities hold true:
\begin{subequations} \label{helicity-discrete}
\begin{align}
 \dt (\Bbfh, \Abfh)
& = -2  \Rmmo\mu (\Hbfh , \jbfh)
= -2  \Rmmo\mumo (\Qbbh \Bbfh, \curlh\Bbfh), \label{helicity-discrete-a} \\
 \dt (\ubfh, \Bbfh )
& = -\Remo\mu\ a(\ubfh, \Hbfh)
    -  \Rmmo (\curl \ubfh, \jbfh)
    + \mu(\fbf, \Hbfh) \nonumber \\
& = -\Remo a(\ubfh, \Qbbh \Bbfh)
    -  \Rmmo\mumo (\curl \ubfh, \curlh\Bbfh )
    + (\fbf, \Qbbh \Bbfh) \label{helicity-discrete-b}.
\end{align}
\end{subequations}
\end{thm}

A consequence of Theorem~\ref{theorem:helicity-discrete} is that
the scheme is helicity-preserving under the assumptions in Remark~\ref{remark:helicity-continuous}.

\begin{rem} \label{remark:lack-helicity-4fields}
The four fields formulation in~\cite{Hu-Ma-Xu:2017} is not helicity-preserving;
see~\cite[Section~3.1]{Hu-Lee-Xu:2021}.
The reason is the presence of ``spurious terms'' appearing in the counterpart of the identities in~\eqref{helicity-discrete} for the four fields formulation.
For instance, given~$\Ibb$ the identity operator, we have
\[
 \dt (\Bbfh, \Abfh)
= -2 \Rmmo (\Qbbh \Bbfh, \curlh\Bbfh)
  + 2 (\ubfh \times \Bbfh, (\Qbbh-\Ibb)\Bbfh) .
\]
The first term on the right-hand sides resembles that in the first equation of~\eqref{helicity-discrete};
the second one is a pollution term measuring the distance of~$\Bbfh$ from~$\Hzcurlh$.
Using instead the seven-fields formulation~\eqref{method-semidiscrete},
the second term vanishes due to the presence
of extra projection terms in the formulation.
\end{rem}

\paragraph*{Well-posedness of the semi-discrete formulation.}
The semi-discrete method~\eqref{method-semidiscrete} is well-posed.
We follow the guidelines of~\cite[Section~5.5]{BeiraodaVeiga-Dassi-Manzini-Mascotto:2023}:
we show that the semi-discrete method can be written
as a first order Cauchy problem with a quartic nonlinearity;
then, the energy bounds in Theorem~\ref{theorem:energy-preservation-discrete}
imply that the nonlinear term is Lipschitz,
whence standard ODE results imply the well-posedness.
\medskip

We begin by taking the time derivative in~\eqref{method-semidiscrete-b}:
\begin{equation} \label{1st-step-wp}
\mu(\dt \jbfh, \kbfh) 
- (\dt\Bbfh,\curl\kbfh) =0
\qquad\qquad
\forall \kbfh \in \Hzcurlh.
\end{equation}
From~\eqref{strong-identities-a} and~\eqref{strong-identities-d},
we further have
\begin{equation} \label{explicit-jbfh}
\jbfh= \Rm \left( \Ebfh + \Qbbh(\ubfh \times \Qbbh \Bbfh) \right).
\end{equation}
We plug this identity in~\eqref{1st-step-wp} and deduce
\begin{equation} \label{2nd-step-wp}
\mu\Rm(\dt\Ebfh, \kbfh)
+ \mu\Rm (\dt \Qbbh(\ubfh\times \Qbbh\Bbfh), \kbfh) 
- (\dt \Bbfh, \curl \kbfh) = 0
\qquad \forall \kbfh \in \Hzcurlh.
\end{equation}
Next, we use~\eqref{2nd-step-wp}
and~\eqref{method-semidiscrete-c} with~$\Cbfh = \curl\kbfh$,
and write
\[
\mu\Rm (\dt\Ebfh, \kbfh)
+ \mu\Rm(\dt \Qbbh(\ubfh\times\Qbbh\Bbfh),\kbfh) 
+ (\curl\Ebfh, \curl\kbfh) = 0
\qquad
\forall \kbfh \in \Hzcurlh.
\]
We condense out the pressure~$\Ph$ from the system
by recalling~\eqref{method-semidiscrete-e},
and restricting test and trial velocity fields
to the space
\[
\Hzbarcurlh
:= \{ \vbfh \in \Hzcurlh \mid (\vbfh, \nabla\Qh) = 0 \quad
                \forall \Qh \in \Hznablah  \}.
\]
Introduce the reduced, discrete test and trial space
\[
\Xbarh := \Hzbarcurlh \times \Hzcurlh \times \Hzdivh .
\]
With this information at hand,
and recalling the strong identities in~\eqref{strong-identities}
and~\eqref{explicit-jbfh},
$\ubfh$, $\Bbfh$, and~$\Ebfh$ in $\Xbarh$
are the solutions to the following Cauchy problem:
for all~$\vbfh$, $\kbfh$, and~$\Cbfh$ in~$\Xbarh$,
\begin{equation} \label{semidiscrete-partial}
\begin{cases}
(\dt \ubfh,\vbfh)
    \!-\! (\ubfh\times \Qbbh(\curl\ubfh), \vbfh) \!+\! \Remo a(\ubfh, \vbfh)
    \!+\! \cL (\jbfh \times \Qbbh\Bbfh,\vbfh) = (\fbf,\vbfh) \\
\mu\Rm (\dt \Ebfh, \kbfh) + (\curl\Ebfh, \curl\kbfh) 
    + \mu\Rm (\dt \Qbbh(\ubfh\times\Qbbh\Bbfh),\kbfh)=0   \\
(\dt\Bbfh, \Cbfh) + (\curl\Ebfh, \Cbfh) = 0 .
\end{cases}
\end{equation}
Inverting the corresponding ``mass'' matrices,
from the first, third, and second identities above,
we deduce that
\begin{itemize}
    \item $\dt\ubfh$ can be interpreted as a cubic function in terms of
    $(\ubfh, \Bbfh, \Ebfh)$;
    \item $\dt\Bbfh$ can be interpreted as a linear function in terms of~$\Ebfh$;
    \item $(\dt \Qbbh(\ubfh\times\Qbbh\Bbfh),\kbfh)$ can be interpreted as a quartic form in terms of $(\ubfh, \Bbfh, \Ebfh)$;
    in fact, we can write
    \[
    \begin{split}
    (\dt \Qbbh(\ubfh\times\Qbbh\Bbfh),\kbfh)
    & = (\dt (\ubfh\times\Qbbh\Bbfh),\kbfh) \\
    & = (\dt \ubfh \times \Qbbh\Bbfh,\kbfh)
        + (\ubfh\times (\dt\Qbbh\Bbfh),\kbfh).
    \end{split}
    \]
    Resorting to the information above on~$\dt\ubfh$ and~$\dt\Bbfh$,
    we can interpret~$\dt\Ebfh$ as a quartic form in terms of $(\ubfh, \Bbfh, \Ebfh)$.
\end{itemize}
In other words, $(\ubfh,\Bbfh,\Ebfh)$ are the solutions
to a first order Cauchy problem
with quartic right-hand side.
On the other hand, Theorem~\ref{theorem:energy-preservation-discrete}
and~\eqref{method-semidiscrete-d}
state that the three fields above are bounded,
thereby entailing a uniform Lipschitz nonlinearity on the right-hand side.
Standard ordinary differential equation theory results
imply the well-posedness of the method; see, e.g., \cite{Layton:2008}.

The well-posedness of the full semi-discrete system~\eqref{method-semidiscrete}
follows from the well-posedness of the reduced semi-discrete system~\eqref{semidiscrete-partial}
and the identities in~\eqref{strong-identities},
in the sense that the four remaining unknowns are derived from the
three solutions to~\eqref{semidiscrete-partial}.

%%%%%%%%%%%%%%%
\section{Error estimates for the semi-discrete scheme} \label{section:convergence-semi-discrete}
%%%%%%%%%%%%%%%
Property~\eqref{discrete-magnetic-zero-divergence}
implies that in system~\eqref{method-semidiscrete}
we seek discrete divergence free magnetic fields,
whence we can replace~$\Hzdivh$ by~$\Hztildedivh$ in~\eqref{div-tilde-space}.
In other words, the total test and trial space~$\Xh$ is replaced by
\[
\Xtildeh
:=  [\Hzcurlh]^5 \times \Hztildedivh \times \Hznablah.
\]
This will be relevant in what follows, since the test field~$\Cbfh$ are divergence free,
which allows us to use the properties
of the operator in~\eqref{projection:divergence-free}.

We prove a fundamental result,
which will be instrumental in deriving the error estimates
in Corollary~\ref{corollary:a-priori-semidiscrete} below.
To this aim,
given~$\BbfI$ and~$\PI$ as in~\eqref{projection:divergence-free}
and~\eqref{definition:Lagrangian-interpolant},
we introduce
\begin{equation} \label{error-type-quantities}
\eubf:= \ubfh-\PiNcalu,
\qquad
\eBbf:= \Bbfh-\BbfI,
\qquad
\eEbf:= \Ebfh-\PiNcalE,
\qquad
\eP  := \Ph  -\PI.
\end{equation}
Given
\[
\chih(\vbf, \Cbf) := \Qbbh (\vbf\times \Qbbh\Cbf)
\qquad\qquad \forall \vbf \in \Hzcurl,
\quad \forall \Cbf \in \Hztildediv,
\]
we further define
\begin{equation} \label{j-quantities}
\ejbf:= \eEbf + \chih(\eubf, \Bbfh),
\qquad\qquad\qquad
\jbfI:= \Rm  \left( \PiNcalE + \chih(\PiNcalu,\Bbfh) \right).
\end{equation}

%%%%
\begin{thm} \label{theorem:partial-convergence}
Consider sequences~$\{\taun\}$ of shape-regular, quasi-uniform meshes.
\footnote{The quasi-uniformity assumption is used
in the estimates for the term~$T_5$
in order to apply Proposition~\ref{proposition:L-infty-stab-PiNcal}.}
Let the solution to~\eqref{MHD-weak} be sufficiently smooth.
Then, there exists a positive constant~$C$ independent of~$\h$
such that, for all~$t$ in~$(0,T]$,
\small{\[
\Norm{\eubf(t)}{}^2 + \Norm{\eBbf(t)}{}^2
+ \int_0^t \Norm{\curl \eubf(s)}{}^2 \ds
+ \int_0^t \Norm{\ejbf(s)}{}^2 \ds
\le C( \Norm{\eubf(0)}{}^2 + \Norm{\eBbf(0)}{}^2 + \h^{2(k+1)}).
\]}\normalsize
The constant~$C$ includes regularity terms of the solution to~\eqref{MHD-weak},
the shape-regularity parameter~$\sigma$ of the mesh,
and the polynomial degree~$k$.
\end{thm}
%%%%%
\begin{proof}
Using standard properties of the cross and scalar products,
and~\eqref{strong-identities-d},
we rewrite~\eqref{method-semidiscrete-a} as
\small{\begin{equation} \label{5.13a}
(\dt \ubfh,\vbfh)
\!-\! (\ubfh\times\omegaboldh,\vbfh)
\!+\! \Remo a(\ubfh,\vbfh)
\!+\! \cL (\jbfh,\vbfh \times \Qbbh\Bbfh)
\!+\! (\nabla\Ph,\vbfh)
\!=\! (\fbf,\vbfh)
\quad
\forall \vbfh \in \Hzcurlh.
\end{equation}}\normalsize
The reduced version of~\eqref{method-semidiscrete-c} reads
\begin{equation} \label{5.13c}
(\dt \Bbfh, \Cbfh) + (\curl \Ebfh, \Cbfh) = 0
\qquad\qquad \forall \Cbfh \in \Hztildedivh .
\end{equation}
Adding and subtracting~$\PiNcalE$ and $\PiNcalu$
in~\eqref{explicit-jbfh}, we deduce
\begin{equation} \label{6.6}
\jbfh 
= \Rm\left( \Ebfh-\PiNcalE + \chih ( \ubfh - \PiNcalu , \Bbfh)
  + \PiNcalE + \chih ( \PiNcalu , \Bbfh) \right).
\end{equation}
We substitute~\eqref{6.6} in~\eqref{5.13a},
add and subtract~$\PiNcalu$ and~$\PI$,
use~\eqref{strong-identities-b},
and get
\small{\begin{equation} \label{6.6.5}
\begin{split}
& (\dt \eubf,\vbfh) - (\eubf \times \omegaboldh,\vbfh) + \Remo a(\eubf,\vbfh) 
 + \cL\Rm (\eEbf + \chih(\eubf,\Bbfh), \chih (\vbfh, \Bbfh)) + (\nabla\eP,\vbfh) \\
& = (\fbf,\vbfh) - (\dt \PiNcalu, \vbfh) + (\PiNcalu \times \Qbbh (\curl \ubfh)),\vbfh) - \Remo a(\PiNcalu,\vbfh) \\
& \quad - \cL\Rm (\PiNcalE + \chih(\PiNcalu,\Bbfh), \chih (\vbfh, \Bbfh)) - (\nabla \PI,\vbfh)
\qquad\qquad \forall \vbfh \in \Hzcurlh.
\end{split}
\end{equation}}\normalsize
Next, we combine~\eqref{method-semidiscrete-b}
and~\eqref{explicit-jbfh},
add and subtract~$\PiNcalE$ and~$\BbfI$,
and get
\begin{equation} \label{6.7}
\begin{split}
& \mu\Rm(\eEbf + \chih(\eubf,\Bbfh), \kbfh) 
        - (\eBbf,\curl \kbfh) \\
& \qquad = - \mu\Rm (\PiNcalE + \chih (\PiNcalu, \Bbfh), \kbfh ) 
        + (\BbfI, \curl \kbfh)
        \qquad \forall \kbfh \in \Hzcurlh.
\end{split}
\end{equation}
Besides, we have
\[
\begin{split}
& (\curl\PiNcalE, \Cbfh)
\overset{\eqref{commuting-operators}}{=}
(\Pitildez(\curl\Ebf), \Cbfh)
\overset{\eqref{def:PiNk}}{=}
(\curl\Ebf, \Cbfh) \\
& \overset{\eqref{MHD-weak-c}}{=}-(\dt\Bbf, \Cbfh)
\overset{\eqref{projection:divergence-free}}{=}
-(\Pitildez(\dt\Bbf), \Cbfh)
= -(\dt (\Pitildez\Bbf), \Cbfh)
\qquad \forall \Cbfh \in \Hztildedivh,
\end{split}
\]
whence we can write
\[
(\Pitildez(\dt\Bbf), \Cbfh) + (\curl\PiNcalE, \Cbfh) = 0
\qquad\qquad \forall \Cbfh \in \Hztildedivh.
\]
Subtracting this to~\eqref{5.13c} entails
\begin{equation} \label{6.8}
(\dt \eBbf, \Cbfh) + (\curl\eEbf , \Cbfh) = 0.
\end{equation}
Recalling~\eqref{trilinear:vanishes},
and taking~$\vbfh=\eubf$, $\kbfh=\eEbf$, and~$\Cbfh=\eBbf$
in~\eqref{6.6.5}, \eqref{6.7}, and~\eqref{6.8},
we arrive at the following set of equations:
\begin{subequations} \label{important-3-equations}
\begin{align}
& ( \dt \eubf,\eubf)
+ \Remo a(\eubf,\eubf) 
+ \cL\Rm (\eEbf + \chih(\eubf,\Bbfh), \chih (\eubf, \Bbfh)) + (\nabla\eP,\eubf) \nonumber \\
& \qquad= (\fbf,\eubf) - (\dt \PiNcalu, \eubf)
        + (\PiNcalu \times \Qbbh (\curl \ubfh)),\eubf)
        - \Remo a(\PiNcalu,\eubf) \label{6.9} \\
& \qquad\qquad - \cL\Rm (\PiNcalE + \chih(\PiNcalu,\Bbfh), \chih (\eubf, \Bbfh))
            - (\nabla \PI,\eubf), \nonumber \\[1em]
\nonumber& \mu\Rm (\eEbf + \chih(\eubf,\Bbfh), \eEbf) - (\eBbf,\curl \eEbf) \\
& \qquad= - \mu\Rm (\PiNcalE + \chih (\PiNcalu, \Bbfh), \eEbf) + (\BbfI, \curl \eEbf), \label{6.10} \\[1em]
& (\dt \eBbf, \eBbf) + (\curl\eEbf , \eBbf) = 0. \label{6.11}
\end{align}
\end{subequations}
Adding~\eqref{6.10} and~\eqref{6.11}, we obtain
\small{\begin{equation} \label{this-identity}
\mu\Rm(\eEbf + \chih(\eubf,\Bbfh), \eEbf)
  + (\dt \eBbf, \eBbf)
= - \mu\Rm(\PiNcalE + \chih(\PiNcalu, \Bbfh), \eEbf)
  + (\BbfI, \curl\eEbf).
\end{equation}}\normalsize
Multiplying~\eqref{this-identity} by~$\cL \mumo$
and adding the resulting identity to~\eqref{6.9} yield
\begin{equation} \label{LHS-RHS}
\LHS = \RHS,
\end{equation}
where
\begin{equation} \label{LHS}
\begin{split}
\LHS
& := (\dt \eubf,\eubf) + \Remo a(\eubf,\eubf) + (\nabla\eP,\eubf)\\
& \qquad   + \cL\Rm (\eEbf + \chih(\eubf,\Bbfh), \eEbf + \chih(\eubf,\Bbfh))
   + \cL \mumo(\dt \eBbf,\eBbf) \\
&  = \frac12 \dt\Norm{\eubf}{}^2 + \Remo \Norm{\curl\eubf}{}^2 +  (\nabla\eP,\eubf)
 + \cL \Rm\Norm{\ejbf}{}^2  + \frac{\cL \mumo}{2} \dt\Norm{\eBbf}{}^2
\end{split}
\end{equation}
and
\begin{equation} \label{RHS}
\begin{split}
 \RHS
& := (\fbf, \eubf) - (\dt\PiNcalu,\eubf) 
    + (\PiNcalu\times \Qbbh(\curl\ubfh),\eubf) 
    - \Remo a(\PiNcalu,\eubf)
    - (\nabla\PI,\eubf)\\
& \qquad    - \cL\Rm (\PiNcalE + \chih(\PiNcalu,\Bbfh), \eEbf 
            + \chih(\eubf,\Bbfh))
            + \cL \mumo(\BbfI,\curl\eEbf) \\
& = (\fbf, \eubf) 
    - (\dt\PiNcalu,\eubf)
    + (\PiNcalu\times \Qbbh(\curl\ubfh),\eubf)
    - \Remo a(\PiNcalu,\eubf)
    - (\nabla\PI,\eubf)\\
& \qquad    - \cL(\jbfI, \ejbf)
            + \cL \mumo (\BbfI,\curl\eEbf). \\
\end{split}
\end{equation}
The term involving~$\eP$ in~$\LHS$, see~\eqref{LHS}, vanishes:
\begin{equation} \label{mixed-vanishes}
(\nabla\eP,\eubf)
\overset{\eqref{error-type-quantities}}{=} (\nabla\eP,\ubfh) - (\nabla\eP,\PiNcalu)
\overset{\eqref{method-semidiscrete-e}}{=} - (\nabla\eP,\PiNcalu)
\overset{\eqref{interpolation:curl}}{=} - (\nabla\eP,\ubf)
\overset{\eqref{MHD-weak-e}}{=} 0.
\end{equation}
To show an upper bound on~$\RHS$,
we rewrite the term involving~$\fbf$ using~\eqref{MHD-weak-a}
with~$\vbfh=\eubf$:
\begin{equation} \label{6.14}
(\fbf,\eubf)
= (\dt \ubf, \eubf) + \Remo a(\ubf, \eubf) + (\nabla \P, \eubf)
  - (\ubf \times (\curl\ubf), \eubf) + \cL (\jbf, \eubf\times \Bbf).
\end{equation}
Observe that
\begin{equation} \label{6.15}
(\BbfI, \curl\eEbf)
\overset{\eqref{projection:divergence-free}}{=} (\Bbf, \curl\eEbf)
\overset{\eqref{MHD-weak-b}}{=} \mu(\jbf,\eEbf).
\end{equation}
Inserting~\eqref{6.14} and~\eqref{6.15} in~\eqref{RHS} yields
\small{\begin{equation} \label{T1-T2-T3-T4-T5}
\begin{split}
\RHS
& =  (\dt(\ubf-\PiNcalu),\eubf) 
        + \Remo a(\ubf-\PiNcalu,\eubf)
        + (\nabla(\P-\PI),\eubf) \\
& \quad + [(\PiNcalu\times \Qbbh(\curl\ubfh),\eubf) - (\ubf\times(\curl\ubf),\eubf)]
        + \cL [ (\jbf, \eEbf + \eubf\times\Bbf) 
        - (\jbfI, \ejbf) ]
= \sum_{j=1}^5 T_j.
\end{split}
\end{equation}}\normalsize
We estimate the five terms on the right-hand side
of~\eqref{T1-T2-T3-T4-T5} in separate steps.

\paragraph*{Estimating~$T_1$.}
Using Cauchy-Schwarz' inequality, Young's inequality,
and Lemma~\ref{lemma:Monk-Nedelec} yields
\begin{equation} \label{T1}
T_1
\le \Norm{\dt (\ubf-\PiNcalu)}{} \Norm{\eubf}{}
\le \overline C_1 \h^{2(k+1)} \Norm{\dt \ubf}{k+1}^2 + \Norm{\eubf}{}^2
=             C_1 \h^{2(k+1)}                        + \Norm{\eubf}{}^2.
\end{equation}

\paragraph*{Estimating~$T_2$.}
Using Cauchy-Schwarz' inequality, Young's inequality
with parameter~$\varepsilon$ to be fixed in~\eqref{choosing-epsilon} below,
and Lemma~\ref{lemma:Monk-Nedelec} yields
\begin{equation} \label{T2}
\begin{split}
T_2
& \le \Remo  \Norm{\curl (\ubf-\PiNcalu)}{} \Norm{\curl\eubf}{}
\le \overline C_2(\varepsilon) \h^{2(k+1)} \Norm{\curl\ubf}{k+1}^2
        + \varepsilon \Remo \Norm{\curl\eubf}{}^2\\
& =  C_2(\varepsilon) \h^{2(k+1)} + \varepsilon \Remo \Norm{\curl\eubf}{}^2.
\end{split}
\end{equation}

\paragraph*{Estimating~$T_3$.}
Using Cauchy-Schwarz' inequality, Young's inequality,
and Lemma~\ref{lemma:Lagrangian-interpolant} yields
\begin{equation} \label{T3}
T_3
\le \Norm{\nabla (\P-\PI)}{} \Norm{\eubf}{}
\le \overline C_3 \h^{2(k+1)} \Norm{\P}{k+2} + \Norm{\eubf}{}^2
\le           C_3 \h^{2(k+1)}                + \Norm{\eubf}{}^2.
\end{equation}

\paragraph*{Estimating~$T_4$.}
We have
\[
T_4= (\PiNcalu\times \Qbbh(\curl\ubfh),\eubf) - (\ubf\times(\curl\ubf),\eubf).
\]
We write
\[
\begin{split}
T_{4}
& = (\PiNcalu \times [\Qbbh (\curl\ubfh)-\curl\ubf],\eubf) 
    + ( [\PiNcalu - \ubf] \times \curl\ubf,\eubf)  \\
& \le \Norm{\PiNcalu}{L^{\infty}(\Omega)}
        \Norm{\curl\ubf - \Qbbh (\curl\ubfh)}{}
        \Norm{\eubf}{}
      + \Norm{\ubf - \PiNcalu}{}
        \Norm{\curl\ubf}{L^{\infty}} 
        \Norm{\eubf}{} .
\end{split}
\]
Note that
\[
\begin{split}
& \Norm{\curl\ubf - \Qbbh (\curl\ubfh)}{}
  \le \Norm{\curl\ubf - \Qbbh (\curl\ubf)}{}
      + \Norm{\curl(\ubf - \ubfh)}{} \\
& \le \Norm{\curl\ubf - \Qbbh (\curl\ubf)}{}
      + \Norm{\curl(\ubf - \PiNcalu)}{}
      + \Norm{\curl\eubf}{}\\
& \overset{\eqref{approximation:L2-proj-Hcurl}, \eqref{PiNcalEbf-estimates}}{\le}
    \widetilde C_{4} \h^{k+1} \Norm{\curl \ubf}{H^{k+1}(\curl,\Omega)}
    + \Norm{\curl\eubf}{}.
\end{split}
\]
Combining the two bounds above, using Young's inequality
with parameter~$\varepsilon$ to be fixed in~\eqref{choosing-epsilon} below,
and recalling Propositions~\ref{proposition:PiNcalEbf-estimates}
and~\ref{proposition:L-infty-stab-PiNcal},
we arrive at
\begin{equation} \label{T4}
T_{4}
\le C_{4} \h^{2(k+1)} 
    + \varepsilon \Norm{\curl \eubf}{}^2 
    + \left( \frac{C_{4}}{\varepsilon} + 1 \right) \Norm{\eubf}{}^2.
\end{equation}

\paragraph*{A preliminary estimate for~$T_5$.}
We are interested in estimating~$\Norm{\jbf-\jbfI}{}$,
where~$\jbfI$ is defined in~\eqref{j-quantities}.
The triangle inequality implies
\begin{equation} \label{pre-T5-1}
\Norm{\jbf-\jbfI}{}
\le \Norm{\Ebf-\PiNcalE}{}
    + \Norm{\ubf\times\Bbf - \Qbbh(\PiNcalu \times\Qbbh\Bbfh)}{} .
\end{equation}
We focus on the second term on the right-hand side:
\begin{equation} \label{pre-T5-2}
\begin{split}
& \Norm{\ubf\times\Bbf - \Qbbh(\PiNcalu \times\Qbbh\Bbfh)}{} 
 \le \Norm{\ubf\times\Bbf - \Qbbh(\ubf \times \Bbf)}{} 
        + \Norm{\Qbbh(\ubf\times\Bbf - \PiNcalu \times\Qbbh\Bbfh)}{} .
\end{split}
\end{equation}
As for the second term on the right-hand side of~\eqref{pre-T5-2},
we exploit the continuity (with constant 1) of~$\Qbbh$ in the $L^2(\Omega)$ norm
and get
\begin{equation} \label{pre-T5-3}
\begin{split}
& \Norm{\Qbbh(\ubf\times\Bbf - \PiNcalu \times\Qbbh\Bbfh)}{} 
    \le \Norm{\ubf\times\Bbf - \PiNcalu \times\Qbbh\Bbfh}{} \\
& \le \Norm{(\ubf- \PiNcalu)\times\Bbf}{}
      + \Norm{\PiNcalu \times (\Bbf-\Qbbh\Bbfh)}{} \\
& \le \Norm{\ubf- \PiNcalu}{} \Norm{\Bbf}{L^\infty(\Omega)}
        + \Norm{\PiNcalu}{L^\infty(\Omega)} \Norm{\Bbf-\Qbbh\Bbfh}{} \\
& \le \Norm{\ubf- \PiNcalu}{} \Norm{\Bbf}{L^\infty(\Omega)}
        + \Norm{\PiNcalu}{L^\infty(\Omega)}
            \Big( \Norm{\Bbf-\Qbbh\Bbf}{}
                + \Norm{\Bbf-\Pitildez\Bbf}{}
                + \Norm{\eBbf}{} \Big) .
\end{split}
\end{equation}
Collecting~\eqref{pre-T5-2} and~\eqref{pre-T5-3} into~\eqref{pre-T5-1},
and using estimates~\eqref{PiNcalEbf-estimates},
\eqref{stability-Linfty-PiNcal} (with the quasi-uniformity of the mesh),
and~\eqref{approximation:L2-proj-Hcurl},
we deduce the existence of a positive~$C_J$ independent of~$\h$
but dependent on the shape-regularity parameter~$\sigma$ of the mesh,
the polynomial degree~$k$,
and the solution to~\eqref{method-semidiscrete} such that
\begin{equation} \label{interpolation-J-estimates}
\Norm{\jbf-\jbfI}{}^2
\le C_J \h^{2(k+1)} + C_J \Norm{\eBbf}{}^2.
\end{equation}

\paragraph*{Estimating~$T_5$.}
We split
\begin{equation} \label{T51-T52}
T_5
= \cL (\jbf, \eEbf + \eubf\times\Bbf-\ejbf) + \cL (\jbf-\jbfI, \ejbf)
=: T_{5,1} +T_{5,2}.
\end{equation}
Cauchy-Schwarz' inequality,
estimate~\eqref{interpolation-J-estimates},
and Young's inequality
with parameter~$\varepsilon$ to be fixed in~\eqref{choosing-epsilon} below entail
\begin{equation} \label{T52}
T_{5,2}
\le \cL \Norm{\jbf-\jbfI}{} \Norm{\ejbf}{}
\le C_{5,2}(\varepsilon) \h^{2(k+1)}
    + \varepsilon \Norm{\ejbf}{}^2
    + C_{5,2}(\varepsilon) \Norm{\eBbf}{}^2.
\end{equation}
Next, we focus on the term~$T_{5,1}$.
We have
\small{\[
\begin{split}
\cL^{-1}T_{5,1}
& = (\jbf, \eubf\times\Bbf - \Qbbh(\eubf\times\Bbf))
  + (\jbf, \Qbbh(\eubf\times\Bbf - \eubf\times\Qbbh\Bbfh))\\
& \le
    (\jbf-\Qbbh\jbf, \eubf\times\Bbf) 
    + \Norm{\jbf}{L^\infty(\Omega)}
        \Norm{\Qbbh(\eubf\times\Bbf - \eubf\times\Qbbh\Bbfh)}{L^1(\Omega)} \\
& \overset{\eqref{stability:L2-in-Linfty}}{\le}        
    (\jbf-\Qbbh\jbf, \eubf\times\Bbf) 
    + C \Norm{\jbf}{L^\infty(\Omega)}
        \Norm{\eubf\times\Bbf - \eubf\times\Qbbh\Bbfh}{L^1(\Omega)}\\
& \le (\jbf-\Qbbh\jbf, \eubf\times\Bbf) 
    + C \Norm{\jbf}{L^\infty(\Omega)}
        \Norm{\eubf}{}
        \Norm{\Bbf - \Qbbh\Bbfh}{} \\
& \le \Norm{\jbf-\Qbbh\jbf}{} \Norm{\eubf}{} \Norm{\Bbf}{L^{\infty}(\Omega)}
    + \widetilde C \Norm{\jbf}{L^\infty(\Omega)}
        \Norm{\eubf}{}
        (\Norm{\Bbf - \Pitildez\Bbf}{}
         + \Norm{\eBbf}{} + \Norm{\Bbf - \Qbbh\Bbf}{}).
\end{split}
\]}\normalsize
We use Young's inequality, estimates~\eqref{Hdiv-approx}
and~\eqref{approximation:L2-proj-Hcurl},
and deduce the existence of positive constants~$\widetilde C_{5,1}$
and~$C_{5,1}$ such that
\begin{equation} \label{T51}
\begin{split}
T_{5,1}
& \le \widetilde C_{5,1} ( \Norm{\jbf-\Qbbh\jbf}{}^2
                            + \Norm{\Bbf-\Qbbh\Bbf}{}^2
                            + \Norm{\Bbf-\Pitildez\Bbf}{}^2
                            + \Norm{\eubf}{}^2
                            + \Norm{\eBbf}{}^2) \\
& \le C_{5,1} \h^{2(k+1)}   + \Norm{\eubf}{}^2 + \Norm{\eBbf}{}^2 .
\end{split}
\end{equation}
Inserting~\eqref{T51} and \eqref{T52} in~\eqref{T51-T52},
we arrive at
\begin{equation} \label{T5}
T_5
\le C_5(\varepsilon) (\h^{2(k+1))} + \Norm{\eubf}{}^2 + \Norm{\eBbf}{}^2)
        + \varepsilon \Norm{\ejbf}{}^2.
\end{equation}

\paragraph*{Collecting the estimates.}
We collect~\eqref{T1}, \eqref{T2}, \eqref{T3}, \eqref{T4}, and~\eqref{T5} in~\eqref{T1-T2-T3-T4-T5},
and obtain
\begin{equation} \label{choosing-epsilon}
\RHS
\le C(\varepsilon) (\h^{2(k+1))} + \Norm{\eubf}{}^2) 
                + C \Norm{\eBbf}{}^2
                + \varepsilon \Norm{\curl \eubf}{}^2
                + \varepsilon \Norm{\ejbf}{}^2.
\end{equation}
Inserting this inequality in~\eqref{LHS-RHS}, using~\eqref{mixed-vanishes},
moving the last two terms on the right-hand side to the left-hand side,
and picking~$\varepsilon$ from Young's inequalities above sufficiently small,
for positive constants~$C_A$ and~$C_B$ only depending
on the data and the shape-regularity parameter~$\sigma$ of the mesh, we write
{ \small\[
\begin{split}
\frac12 \dt\Norm{\eubf}{}^2
 + \Remo \Norm{\curl\eubf}{}^2
 +  (\nabla\eP,\eubf)
 + \Norm{\ejbf}{}^2
 + \frac12 \dt\Norm{\eBbf}{}^2
\le C_A \h^{2(k+1)}
+ C_B (\Norm{\eubf}{}^2 + \Norm{\eBbf}{}^2) .
 \end{split}
\]}
We integrate in time and get
\[
\begin{split}
& \Norm{\eubf(t)}{}^2 
 + \Norm{\eBbf(t)}{}^2
 + \int_0^t \Norm{\curl \eubf(s)}{}^2 \ds
 + \int_0^t \Norm{\ejbf(s)}{}^2 \ds \\
& \quad \le \Norm{\eubf(0)}{}^2 
    + \Norm{\eBbf(0)}{}^2
    + \int_0^t C_A \h^{2(k+1)} \ds
    + \int_0^t C_A [\Norm{\eubf(s)}{}^2 + \Norm{\eBbf(s)}{}^2] \ds
\qquad \forall t \in (0,T].
\end{split}
\]
Applying Gronwall's inequality
\[
u(t) \le \alpha(t) + \int_0^t \beta(s) u(s) \ds
\quad \Longrightarrow \quad
u(t) \le  \alpha(t) + \int_0^t \alpha(s)\beta(s) \exp{\left(\int_s^0 \beta(r)\dr\right)} \ds,
\]
the assertion follows.
\end{proof}

A consequence of Theorem~\ref{theorem:partial-convergence} is ancillary
for proving error estimates for the semi-discrete scheme~\eqref{method-semidiscrete}.
%%%%
\begin{cor} \label{corollary:a-priori-semidiscrete}
Consider sequences~$\{\taun\}$ of shape-regular, quasi-uniform meshes.
Let the solution to~\eqref{MHD-weak} be sufficiently smooth.
Then, there exists a positive constant~$C$ independent of~$\h$
such that, for all~$t$ in~$(0,T]$,
\[
\begin{split}
& \Norm{(\ubf-\ubfh)(t)}{}
+ \Norm{(\Bbf-\Bbfh)(t)}{}
+ \Big(\int_0^t \Norm{\curl (\ubf-\ubfh)(s)}{}^2 \ds \Big)^\frac12
+ \Big( \int_0^t \Norm{(\jbf-\jbfh)(s)}{}^2 \ds \Big)^\frac12 \\
& \le C \Big(  \Norm{\eubf(0)}{} + \Norm{\eBbf(0)}{} + \h^{k+1} \Big).
\end{split}
\]
The constant~$C$ includes regularity terms of the solution to~\eqref{MHD-weak},
the shape-regularity parameter~$\sigma$ of the mesh,
and the polynomial degree~$k$.
\end{cor}
%%%
\begin{proof}
The triangle inequality implies
\[
\begin{split}
& \Norm{(\ubf-\ubfh)(t)}{}
+ \Norm{(\Bbf-\Bbfh)(t)}{}
+ \Big(\int_0^t \Norm{\curl (\ubf-\ubfh)(s)}{}^2 \ds \Big)^\frac12
+ \Big( \int_0^t \Norm{(\jbf-\jbfh)(s)}{}^2 \ds \Big)^\frac12 \\
& \le C \Big[\Norm{(\ubf-\PiNcalu)(t)}{}
    + \Norm{(\Bbf-\Pitildez \Bbf)(t)}{}
    + \Big(\int_0^t \Norm{\curl (\ubf-\PiNcalu)(s)}{}^2 \ds \Big)^\frac12\\
& \qquad\qquad  + \Big( \int_0^t \Norm{(\jbf-\jbfI)(s)}{}^2 \ds \Big)^\frac12 
 + \Norm{\eubf(t)}{} 
 + \Norm{\eBbf(t)}{} \\
& \qquad\qquad + \Big( \int_0^t \Norm{\curl \eubf(s)}{}^2 \ds \Big)^\frac12
 + \Big( \int_0^t \Norm{\ejbf(s)}{}^2 \ds \Big)^\frac12 \Big] .
\end{split}
\]
The assertion follows combining Theorem~\ref{theorem:partial-convergence},
and estimates~\eqref{PiNcalEbf-estimates},
\eqref{Hdiv-approx},
and~\eqref{interpolation-J-estimates}.
\end{proof}
%%%%%

Using estimates~\eqref{PiNcalEbf-estimates}
and~\eqref{Hdiv-approx} to handle the initial data error
implies convergence for the velocity, magnetic field, vorticity,
and electric density unknowns.
Convergence for the other three variables can be deduced
from the relations in~\eqref{strong-identities}.
Computational tests can be found in~\cite{Hu-Lee-Xu:2021}
and are in agreement with the theoretical findings above.

\paragraph*{Acknowledgements.}
LBdV and LM have been partially funded by the European Union
(ERC, NEMESIS, project number 101115663).
Views and opinions expressed are however those of the author(s) only
and do not necessarily reflect those of the EU or the ERC Executive Agency.
LM has been also partially funded by MUR (PRIN2022 research grant n. 202292JW3F).
LM and LBdV are members of the Gruppo Nazionale Calcolo Scientifico-Istituto
Nazionale di Alta Matematica (GNCS-INdAM).
The work of KH was supported by a Royal Society University Research Fellowship
(URF$\backslash$R1$\backslash$221398).

% ----------------------------
% Bibliography
% ----------------------------

{\footnotesize
\bibliography{bibliogr}
\bibliographystyle{plain}
}

\appendix
% --------------------------------------------------
\section{Proof of Theorem~\ref{theorem:helicity-discrete}} \label{appendix:helicity}
% --------------------------------------------------

\noindent\textbf{Proof of~\eqref{helicity-discrete-a}.}
The fact that~$\curl\Abfh=\Bbfh$,
the chain rule, and integrating by parts
imply
\small{\begin{equation} \label{dtBhAh}
\dt (\Bbfh,\Abfh)
= (\curl\dt\Abfh,\Abfh) + (\dt\Abfh,\Bbfh)
= (\dt\Abfh,\curl\Abfh) + (\dt\Abfh,\Bbfh)
= 2 (\dt\Abfh,\Bbfh).
\end{equation}}\normalsize
From~\eqref{relation-Bh-curlEh}, we have~$\dt\Bbfh=-\curl\Ebfh$.
Since~$\dt \Bbfh = \curl (\dt\Abfh)$ by definition,
we deduce~$\curl(\dt\Abfh+\Ebfh) =0$.
This entails the existence of~$\Phih$ in~$\Hznablah$ such that
\[
\dt\Abfh = - \Ebfh - \nabla \Phih .
\]
Testing this identity with~$\Bbfh$,
and using an integration by parts
and the fact that~$\div\Bbfh=0$ for all times,
see~\eqref{discrete-magnetic-zero-divergence},
we can write
\begin{equation} \label{dtAh-Eh}
(\dt\Abfh,\Bbfh)
= - (\Ebfh+\nabla\Phih, \Bbfh)
= - (\Ebfh,\Bbfh).
\end{equation}
On the other hand, we have
\begin{equation} \label{rewriting-Eh}
\Ebfh
\overset{\eqref{strong-identities-a},\eqref{strong-identities-d}}{=}
    \Rmmo \jbfh - \Qbbh(\ubfh \times \Qbbh\Bbfh).
\end{equation}
Consequently, we write
\[
\begin{split}
(\Ebfh,\Bbfh)
& = \Rmmo (\Bbfh,\jbfh) - (\Qbbh(\ubfh\times\Qbbh\Bbfh), \Bbfh) \\
& \overset{\eqref{L2-projection-curl}}{=}
    \Rmmo (\Bbfh,\jbfh) - (\ubfh\times\Qbbh\Bbfh, \Qbbh \Bbfh)
  \overset{\eqref{L2-projection-curl},\eqref{trilinear:vanishes}}{=}
  \Rmmo (\Qbbh\Bbfh,\jbfh)
  \overset{\eqref{strong-identities-d}}{=} \Rmmo \mu (\Hbfh,\jbfh).
\end{split}
\]
Plugging this into~\eqref{dtAh-Eh},
then using the resulting identity in~\eqref{dtBhAh} yields~\eqref{helicity-discrete-a}.
\medskip

\noindent \textbf{Proof of~\eqref{helicity-discrete-b}.}
Observe that
\[
-(\jbfh \times \Hbfh, \Qbbh \Bbfh)
\overset{\eqref{strong-identities-c},\eqref{strong-identities-d}}{=}
- \mu^{-2}  ((\curlh\Bbfh) \times \Qbbh\Bbfh, \Qbbh \Bbfh)
\overset{\eqref{trilinear:vanishes}}{=} 0.
\]
We take~$\vbfh = \Qbbh \Bbfh$ in~\eqref{method-semidiscrete-a},
use~\eqref{strong-identities-b}
and the properties of the cross product,
and get
\begin{equation} \label{dtuhBh}
(\dt\ubfh,\Bbfh)
+ ((\Qbbh(\curl\ubfh))\times\ubfh, \Qbbh\Bbfh)
+ \Remo a(\ubfh, \Qbbh\Bbfh)
+ (\nabla\Ph, \Qbbh \Bbfh)
= (\fbf, \Qbbh \Bbfh).
\end{equation}
Using~\eqref{rewriting-Eh}, we deduce
\begin{equation} \label{dtBh}
\dt\Bbfh
\overset{\eqref{relation-Bh-curlEh}}{=}
-\curl\Ebfh
\overset{\eqref{rewriting-Eh}}{=}
- \Rmmo \curl \jbfh + \curl\Qbbh (\ubfh \times \Qbbh\Bbfh).
\end{equation}
On the other hand, the exact sequence's properties imply
\[
- (\nabla\Ph, \Qbbh\Bbfh) 
\overset{\eqref{L2-projection-curl}}{=} - (\nabla\Ph, \Bbfh)
= (\Ph, \div\Bbfh)
\overset{\eqref{discrete-magnetic-zero-divergence}}{=}0.
\]
We arrive at
\[
\begin{split}
\dt(\ubfh,\Bbfh)
& = (\dt\ubfh,\Bbfh) + (\ubfh, \dt\Bbfh) \\
& \overset{\eqref{dtuhBh},\eqref{dtBh}}{=}
   - (\Qbbh(\curl\ubfh)\times\ubfh, \Qbbh\Bbfh) 
        - \Remo a(\ubfh, \Qbbh\Bbfh)
        + (\fbf, \Qbbh\Bbfh)\\
& \qquad\qquad - \Rmmo (\ubfh, \curl\jbfh)
                + (\ubfh, \curl\Qbbh(\ubfh\times\Qbbh\Bbfh)) .
\end{split}
\]
The cross product's properties entail
\[
- (\Qbbh(\curl\ubfh)\times\ubfh, \Qbbh\Bbfh) 
 + (\ubfh, \curl\Qbbh(\ubfh\times\Qbbh\Bbfh)) =0.
\]
Identity~\eqref{helicity-discrete-b} follows
combining the two equations above.
\end{document}